\documentclass{article}
\usepackage{graphicx} 
\usepackage{url}
\usepackage{amsmath,amsthm,amsfonts,amssymb}
\usepackage[colorlinks=true,citecolor=black,linkcolor=black,urlcolor=blue]{hyperref}
\usepackage{cleveref}
\usepackage{float}
\usepackage{enumitem}
\usepackage{todonotes}
\usepackage{diagbox}
\usepackage{subcaption}
\usepackage{authblk}

\newtheorem{theorem}{Theorem}
\newtheorem{corollary}[theorem]{Corollary}
\newtheorem{lemma}[theorem]{Lemma}
\newtheorem{proposition}[theorem]{Proposition}
\newtheorem{conjecture}[theorem]{Conjecture}
\newtheorem{remark}[theorem]{Remark}
\newtheorem{question}[theorem]{Question}

\usepackage{xcolor}

\newcommand{\dist}[0]{\operatorname{dist}}
\newcommand{\oct}[0]{\operatorname{oct}}
\newcommand{\colo}[2]{(1^{#1}, 2^{#2})}
\newcommand{\ex}[3]{\operatorname{ex}^{#1}_{#2}\hspace{-3pt}{#3}}

\title{Between proper and square coloring of planar graphs, hardness and extremal graphs}
\author[]{Thomas Del\'epine}
\affil[]{LIRMM, Univ Montpellier, CNRS, Montpellier, France}
\date{}
\begin{document}

\maketitle
\begin{abstract}
{\sc$\colo{a}{b}$-coloring} is the problem of partitioning the vertex set of a graph into $a$ independent sets and $b$ 2-independent sets. This problem was recently introduced by Choi and Liu. We study the computational complexity and extremal properties of {\sc$\colo{a}{b}$-coloring}. We prove that this problem is NP-Complete even when restricted to certain classes of planar graphs, and we also investigate the extremal values of $b$ when $a$ is fixed and in some $(a + 1)$-colorable classes of graphs. In particular, we prove that $k$-degenerate graphs are $\colo{k}{O(\sqrt{n})}$-colorable, that triangle-free planar graphs are $\colo{2}{O(\sqrt{n})}$-colorable and that planar graphs are $\colo{3}{O(\sqrt{n})}$-colorable. All upper bounds obtained are tight up to a constant factor.
\end{abstract}

\section{Introduction}

The \emph{chromatic number} of a graph $G$, denoted $\chi(G)$ is the least integer $k$ such that the vertex set of $G$ can be partitioned into $k$ independent sets. The problem of computing the chromatic number of a graph is one of the oldest and most widely studied problems in graph theory, from both an algorithmic and from a structural point of view. Those two points of views lead to very famous and celebrated results, such as the NP-Completeness of {$k$-coloring}~\cite{Karp1972} even when $k = 3$ and in planar graphs~\cite{GAREY1976237}, or the Four Color Theorem~\cite{appel1976every, appel1977every, ROBERTSON19972} stating that the chromatic number of every planar graph is at most $4$. However, the only known proofs of the latter result are computer-assisted due to the number of cases to study, and no combinatorial proof of the Four Color Theorem verifiable without computer assistance is known. In this direction, there has been recent interest regarding variants of the Four Color Theorem, and more precisely, strengthenings of the Five Color Theorem. For instance, it is proven in~\cite{inoue20255} that it is possible to $5$-color any planar graph of order $n$ such that at least one color class is of size at most $n/6$, which is better than the $n/5$ bound provided by the pigeon hole principle. Although this statement is much weaker than the Four Color Theorem, the proof is not computer-assisted. The Four Color Theorem also provides the best possible lower-bound for the size of the largest independent set in a planar graph $G$ with $\alpha(G) \geq n/4$ which is attained for a disjoint union of $K_4$. There is no known proof of this bound that is independent of the Four Color Theorem, and the best known lower-bound independent of the Four Color Theorem is $\alpha(G) \geq 3n/13$ provided in~\cite{cranston2016planargraphsindependenceratio}.  

A variant of the chromatic number is the \emph{$2$-distance chromatic number}, denoted $\chi_2(G)$, introduced by Kramer and Kramer in 1969 in \cite{kramer1969probleme}, and defined as the least integer $k$ such that the vertex set of $G$ can be partitioned into $k$ $2$-independent sets. The decision problem asking, for a graph $G$, whether or not $\chi_2(G) \leq k$ is called {\sc $2$-distance $k$-coloring} and it was proven NP-complete for every $k \geq 4$ in~\cite{sharp2007distance}. In fact, a much more general result is provided in~\cite{sharp2007distance} regarding {\sc$d$-distance $k$-coloring}, which is NP-Complete if $k > \lfloor 3d/2 \rfloor$ and Polynomial if $k \leq \lfloor 3d/2 \rfloor$. A conjecture of Wegner from 1977 in~\cite{wegner1977graphs} about the value of $\chi_2$ in planar graphs has generated substantial interest:

\begin{conjecture}[Wegner~\cite{wegner1977graphs}, 1977]
For every planar graph $G$,
\begin{align*}
    &\bullet~\chi_2(G) \leq 7 & \textrm{if } \Delta(G) = 3, \\
    &\bullet~\chi_2(G) \leq \Delta(G) + 5 & \textrm{if } 4 \leq \Delta(G) \leq 7, \\
    &\bullet~\chi_2(G) \leq \lfloor 3\Delta(G)/2\rfloor + 1 & \textrm{otherwise.}
\end{align*}
\end{conjecture}

Notice that this conjecture, if true, is tight.
Since then, it has been proven in~\cite{amini2013unified} and in~\cite{havet2017listcolouringsquaresplanar} that for every graph $G$ of maximum degree $\Delta$, $\chi_2(G) \leq (1 + o(1))3\Delta/2$ as $\Delta$ tends to infinity. Other results focus on bounding $\chi_2$ not in planar graphs of bounded maximum degree but in planar graphs of bounded maximum average degree~\cite{bonamy20142, la20252}. A simple lower bound on $\chi_2$ is, for every graph $G$, $\chi_2(G) \geq \Delta(G) + 1$. There has been a lot of interest regarding when this lower bound is tight since a conjecture of Wang and Lih from 2003 in~\cite{Wang_and_Lih}:

\begin{conjecture}[Wang and Lih~\cite{Wang_and_Lih}, 2003]
    For every integer $k \geq 5$, there exists an integer $M(k)$ such that every planar graph with girth at least $k$ and maximum degree at least $M(k)$ satisfies $\chi_2(G) = \Delta(G) + 1$.
\end{conjecture}
This conjecture has been validated for $k \geq 7$ in~\cite{Borodin2004, Borodin2004_2, DVORAK2008838, DVORAK20092634} and proven wrong for the cases $k = 5$ and $k = 6$, although planar graphs of girth at least $5$ are $2$-distance colorable with at most $\Delta + 2$ colors~\cite{BONAMY2019218}. We refer the reader to \cite{cranston2022coloring} for a recent survey regarding colorings at distance-2.

Due to the interest and importance of both $\chi$ and $\chi_2$, it is natural to study parameters that are between proper coloring and $2$-distance coloring. Recently, Choi and Liu introduced in~\cite{choi2025propersquarecoloringssparse} the notion of $\colo{a}{b}$-coloring that aims at partitioning a graph into at most $a$ independent sets and at most $b$ $2$-independent sets. In other words, $\colo{a}{b}$-coloring aims at coloring a graph with $a$ \emph{distance-1 colors} (two vertices at distance 1 from each other cannot have the same distance-1 color) and $b$ \emph{distance-2 colors} (two vertices at distance at most 2 from each other cannot have the same distance-2 color). When $b = 0$, this notion is exactly $a$-coloring, and when $a$ equals $0$, this notion is exactly $2$-distance $b$-coloring. In their work, Choi and Liu focus on the $\colo{a}{b}$-coloring (and the \emph{list coloring} versions of $\colo{a}{b}$-coloring) of sparse graphs, and, for a fixed value of $a$, they provide bounds on the value of $b$ for graphs of bounded maximum average degree or of bounded genus to be $\colo{a}{b}$-colorable. Among other results, they prove:  

\begin{theorem}[Choi and Liu~\cite{choi2025propersquarecoloringssparse}, 2025]\label{thm:choi_and_liu}
    For every planar graph $G$,
    \begin{itemize}
        \item if $G$ is of girth at least $7$, then $G$ is $\colo{2}{12}$-colorable, and
        \item if $G$ is of girth at least $8$, then $G$ is $\colo{2}{2}$-colorable, and
        \item if $G$ is of girth at least $10$, then $G$ is $\colo{2}{1}$-colorable.
    \end{itemize}
\end{theorem}

They also construct a planar graph of girth $7$ that is not $\colo{2}{1}$-colorable and planar graphs of girth $6$ that are not $\colo{2}{k}$-colorable, for every $k$.

Building on their work, we continue the study of $\colo{a}{b}$-coloring. In Section~\ref{sec:NPC}, we provide proofs that {\sc $\colo{a}{b}$-coloring}, the decision problem asking whether a given graph is $\colo{a}{b}$-colorable, is NP-Complete even when restricted to various subclasses of planar graphs. More precisely, our hardness results are of three kinds. We provide classical NP-hardness proofs stating that for some values of $a$ and $b$, {\sc $\colo{a}{b}$-coloring} is NP-Complete. We also provide conditional NP-Completeness statements stating that for some values of $a$ and $b$, {\sc $\colo{a}{b}$-coloring} is NP-Complete in some restricted classes of graphs if and only if there exists a graph with some structural constraints that is not $\colo{a}{b}$-colorable. Finally, whenever possible, we provide gap statements. Let $\mathcal{S}$ and $\mathcal{S}'$ be two sets such that $\mathcal{S} \subseteq \mathcal{S}'$ and such that the problem of deciding if some element $e$ belongs to $\mathcal{S}$ is NP-Complete. Then we say that there is a gap between $\mathcal{S}$ and $\mathcal{S}'$ if deciding if some element $e$ belongs to $\mathcal{S}$ is NP-Complete even with the additional information that either $e \in \mathcal{S}$ or $e \not\in \mathcal{S}'$. Notice that such a statement is stronger than just saying that deciding if some element $e$ belongs to $\mathcal{S}'$ is NP-Complete. More precisely, our results are the following:
\begin{itemize}
    \item for every $k \in \{2, 3\}$ and $g \geq 3$, {\sc $\colo{1}{k}$-coloring} is NP-Complete even when restricted to bipartite planar graphs with maximum degree $k + 1$ and girth at least $g$ (Theorem~\ref{thm:npc_1_2_coloring_for_every_g} and Theorem~\ref{thm:npc_1_3_coloring_for_every_g}), and
    \item for every $k \geq 4$, {\sc $\colo{1}{k}$-coloring} is NP-Complete even when restricted to bipartite planar graphs with maximum degree $k + 1$ (Theorem~\ref{thm:1_k_npc}
    ), and
    \item for every $k \geq 1$, and for every class $\mathcal{C}$ containing disjoint unions of $P_2$, $P_5$ and $P_8$, deciding if the vertex set of a planar graph can be partitioned into a set inducing a graph from $\mathcal{C}$ and $k$ 2-independent sets or is not $\colo{2}{k}$-colorable is NP-Complete (Corollary~\ref{coro:2k_gap_strutrural}), and
    \item for every $k \geq 1$ and $g \in \{3, 4, 5, 6\}$, there exists a finite set $\mathcal{S}$ of bipartite planar graphs of girth at least $g$ such that for every class of graphs $\mathcal{C}$ containing the disjoint unions of the graphs in $\mathcal{S}$, the problem of deciding whether the vertex set of a planar graph of girth at least $g$ can be partitioned into a set of vertices inducing graph from $\mathcal{C}$ and $k$ 2-independent sets or is not $\colo{2}{k}$-colorable is NP-Complete (Corollary~\ref{coro:2k_np_hard_gorth_3456}), and
    \item deciding if a planar graph is $3$-colorable or not $\colo{3}{1}$-colorable is NP-Complete (Theorem~\ref{thm:npc_gap_3_1}), and 
    \item for every $k \geq 1$, {\sc $\colo{3}{k}$-coloring} is NP-Complete even when restricted to planar graphs of maximum degree $3k + 4$ (Theorem~\ref{thm:npc_3_k}).
\end{itemize}

 Then, in Section~\ref{sec:extr_values}, we study the extremal values of $\colo{a}{b}$-colorings. For some class of graphs $\mathcal{C}$, we denote by $\ex{a}{\mathcal{C}}{(n)}$ the smallest integer $b$ such that every graph of order $n$ in $\mathcal{C}$ is $\colo{a}{b}$-colorable, and we try to find lower and upper bounds on $\ex{a}{\mathcal{C}}{(n)}$ for classes $\mathcal{C}$ of $(a + 1)$-colorable graphs for which $\ex{a}{\mathcal{C}}{(n)}$ is sub-linear. We managed to prove that:

 \begin{itemize}
     \item every $k$-degenerate graph is $\colo{k}{O(\sqrt{n})}$-colorable (Theorem~\ref{thm:coloring_k_degenerate_graphs}), and
     \item every planar graph of girth at least $4$ is $\colo{2}{O(\sqrt{n})}$-colorable (Theorem~\ref{thm:2_k-colo_planar_girth_4}), and
     \item every planar graph is $\colo{3}{O(\sqrt{n})}$-colorable (Theorem~\ref{thm:bound_sqrt_planar}). Additionally, the proof is independent of the Four Color Theorem.
 \end{itemize}
 Moreover, all the previous bounds are tight up to a constant factor. 

\section{NP-Completeness}\label{sec:NPC}

Let {\sc $\colo{a}{b}$-coloring} be the problem of partitioning the vertex set of a graph into $a$ independent sets and $b$ $2$-independent sets. In this section, we will prove hardness results for some variants of {\sc$\colo{a}{b}$-coloring}. The map associating a color to each vertex has polynomial size, and the verification procedure is a simple algorithm that checks that every vertex is well colored, thus it has polynomial time complexity. Therefore, for every integer $a$ and $b$, {\sc$\colo{a}{b}$-coloring} is in NP ($a$ and $b$ being part of the input or not). So, in what follows, in every NP-Completeness proof, it is sufficient to prove NP-hardness. 
Through our reductions, we use different NP-Complete problems defined as follows :
\begin{itemize}
    \item {\sc 3-coloring} : Given a graph $G$, is $G$ $3$-colorable ?
    \item {\sc $(\Delta_1, \Delta_1)$-coloring} : Given a graph $G$, can $G$ be $2$-colored such that each color class induces a graph of maximum degree $1$ ?
    \item {\sc restricted planar 3-sat} : Given a $3$-CNF SAT formula $\varphi$ such that
    \begin{itemize}
     \item each clause has size 2 or 3,
     \item each variable appears exactly twice positively and once negatively,
     \item the variable-clause incidence graph is planar.
    \end{itemize}
    Is $\varphi$ satisfiable ?i
\end{itemize}

\begin{theorem}\label{thm:NPC-problems} The following algorithmic problems are NP-Complete :
    \begin{itemize}
        \item  {\sc$3$-coloring}, even when restricted to planar graphs of maximum degree $4$ (see~\cite{GAREY1976237}).
        \item  {\sc $(\Delta_1, \Delta_1)$-coloring}, even when restricted to triangle-free planar graphs of maximum degree $4$ (see~\cite{montassier2013near}).
        \item  {\sc restricted planar 3-sat} (see~\cite{DJPSY}).
    \end{itemize}
\end{theorem}

The set of graphs $\colo{2}{0}$-colorable is the set of bipartite graphs, the set of graphs $\colo{1}{1}$-colorable is the set of disjoint unions of starts, the set of $\colo{0}{2}$-colorable graphs is the set of disjoint union of $K_1$ and $K_2$ and the set of $\colo{0}{3}$-colorable graphs is the set of disjoint union of paths and cycles. Therefore, in all these cases {\sc$\colo{a}{b}$-coloring} is in P. On the other hand, for $k \geq 4$, {\sc$\colo{0}{k}$-coloring} is known to be NP-hard (see~\cite{FEDER2021103210}). In the following subsections, we will prove that {\sc$\colo{a}{b}$-coloring} is NP-Complete for the remaining values of $a$ and $b$ with $a \leq 3$ in some restricted subclasses of planar graphs.

\subsection{{\sc $\colo{1}{2}$-coloring} in planar graphs of large girth}

\begin{figure}
    \centering
    \includegraphics[width=0.8\linewidth]{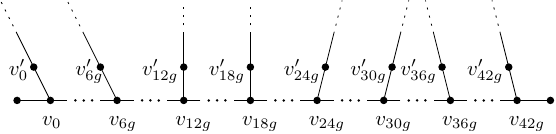}
    \caption{The vertex gadget used in the proof of Theorem~\ref{thm:npc_1_2_coloring_for_every_g}.}
    \label{fig:vertex_gadget_for_1_2_degree3_NPC}
\end{figure}

In what follows, we say that we connect two vertices $u$ and $v$ with a path of length $k$ whenever we add $k - 1$ new vertices $s_1, \dots, s_{k-1}$ and the edges $us_1$, $s_{k - 1}v$ and $s_is_{i + 1}$ for every $i < k - 1$.
\begin{theorem}\label{thm:npc_1_2_coloring_for_every_g}
    For every $g \geq 3$, {\sc $\colo{1}{2}$-coloring} is NP-Complete even when restricted to bipartite planar graphs with maximum degree $3$ and girth at least $g$.
\end{theorem}

\begin{proof}
    We reduce {\sc $(\Delta_1, \Delta_1)$-coloring}, which is known to be NP-Complete on triangle-free planar graphs with maximum degree 4 by Theorem~\ref{thm:NPC-problems}. 
    Let $G$ be a triangle-free planar graph of maximum degree $4$. We construct from $G$ the graph $G'$ by adding one path graph $v_0, \dots, v_{42g}$ per vertex $v \in V(G)$. Then, for every $uv \in E(G)$, we chose $i$, $j \in \{0, 1, 2, 3\}$ so that we can connect $u_{12ig}$ and $u_{(12i + 6)g}$ to $v_{12jg}$ and $v_{(12j + 6)g}$ respectively with two paths of length $9$ in such a way that at the end of the procedure, $G'$ is still planar. Finally, for every vertex $v \in V(G)$ and for every integer $i \in \{0, \dots, 14g\}$, if $v_{3i}$ is of degree $2$ then we add a pendant edge to it. So for every vertex $v \in V(G)$, $v$ is represented in $G'$ by the graph from Figure~\ref{fig:vertex_gadget_for_1_2_degree3_NPC}. Observe that $G'$ is planar, bipartite, has maximum degree $3$, and girth at least $12g + 18$.
    \begin{itemize}
        \item If $G$ is $(\Delta_1, \Delta_1)$-colorable, then let us fix a $(\Delta_1, \Delta_1)$-coloring of $G$ and for every vertex $v \in V(G)$, we copy the color of $v$ in the $(\Delta_1, \Delta_1)$-coloring of $G$ to color $v_0, v_3, \dots, v_{42g}$ in $G'$, considering the $2$ used colors as distance-2 colors. We claim that this partial $\colo{0}{2}$-coloring of $G'$ can be extended into a $\colo{1}{2}$-coloring of $G'$. Indeed, for every vertex $v \in V(G)$ colored $c$, if no neighbor of $v$ in $G$ is also colored $c$, then we can color every vertex of degree $1$ in $G'$ with the distance-1 color, $v_1, v_4, \dots, v_{42g - 2}$ with the distance-1 color, $v_2, v_5, \dots, v_{42g - 1}$ with the other distance-2 color and $v'_0, v'_{6g}, \dots, v'_{42g}$ with the distance-1 color. Otherwise, let $u$ be the only neighbor of $v$ in $G$ also colored $c$, and assume without loss of generality that $u_{12ig}$ is connected to $v_{12jg}$ and that $u_{12i'g}$ is connected to $v_{12j'g}$ for some $i$, $j$, $i'$, $j'$. In this case, we can color $u_{12ig}'$ and $v_{12j'g}'$ with the distance-1 color and $v_{12jg}'$ $u_{12i'g}'$ with the other distance-2 color. The remaining vertices can be easily colored so that we obtain a $\colo{1}{2}$-coloring of $G'$.
        \item If $G'$ is $\colo{1}{2}$-colorable, then let us fix a $\colo{1}{2}$-coloring. Notice that for every vertex $v \in V(G)$ and for every $i$, $j \in \{0, 3, \dots, 42g\}$ $v_i$ and $v_j$ must have the same distance-2 color. Moreover, at most one of $v'_0, v'_{6g}, \dots, v'_{42g}$ can have the other distance-2 color since once $v'_{6gi}$ (for some $i$) is colored with a distance-2 color, it forces the color of $v'_{6gj}$ for every $j \neq i$. 
        Let $uv \in E(G)$ and let $i$, $j$, $i'$, $j'$ be such that $u_{12ig}$ and $u_{12i'g}$ are connected to $v_{12jg}$ and $v_{12j'g}$ receptively. Then if $u_{0}$ and $v_{0}$ have the same distance-2 color then one of $u'_{12ig}$ and $u'_{12i'g}$ must have the other distance-2 color, and similarly for $v'_{12jg}$ and $v'_{12j'g}$ since a path  of length $9$ is not $\colo{1}{2}$-colorable if its two endpoints are precolored with the same distance-2 color and the neighbor of each endpoint is precolored with the distance-1 color. 
        
        Therefore, for every $u$, $w \in N_G(v)$ with $u \neq w$, if $u_0$ and $v_0$ have the same distance-2 color then $w_0$ and $v_0$ must have different distance-2 colors and therefore $u_0$ and $w_0$ have distinct distance-2 colors. So, we can $(\Delta_1, \Delta_1)$-color $G$ by copying, for every vertex $v \in V(G)$, the color of $v_0$.
    \end{itemize}
    Therefore, $G$ is $(\Delta_1, \Delta_1)$-colorable if and only if $G'$ is $\colo{1}{2}$-colorable. So {\sc $\colo{1}{2}$-coloring} is NP-Complete, even when reduced to bipartite planar graphs of maximum degree $3$ and girth at least $g$ for every $g \geq 3$.
\end{proof}

\subsection{{\sc $\colo{1}{3}$-coloring} in planar graphs of large girth}

\begin{theorem}\label{thm:npc_1_3_coloring_for_every_g}
    For every integer $g\geq 3$, {\sc $\colo{1}{3}$-coloring} is NP-Complete on bipartite planar graphs with maximum degree 4 and girth at least $g$.
\end{theorem}

\begin{proof}
    We reduce {\sc 3-coloring}, which is known to be NP-Complete even in planar graphs of maximum degree $4$ by Theorem~\ref{thm:NPC-problems}. Let $G$ be a planar graph of maximum degree $4$. 
    We construct from $G$ the graph $G'$ by adding one path graph $v_0\dots v_{18g}$ per vertex $v \in V(G)$. Then, for every edge $uv \in E(G)$, we chose $i$, $j \in \{0, 1, 2, 3\}$ so that we can connect $u_{6i}$ to $v_{6j}$ by a path of length $2$ in such a way that at the end of the procedure, $G'$ is still planar. 
    Finally, for every $v \in V(G)$ and for every $i \in \{0, \dots, 18g\}$, if $v_i$ is of degree less than $4$ in $G'$ then we add pendant edges so that $v_i$ has degree $4$ in $G'$. Observe that $G'$ is planar, bipartite, has maximum degree $4$ and girth at least $18g + 6$.
    \begin{itemize}
        \item If $G$ is $3$-colorable, then assume the three colors are $1$, $2$ and $3$ and let us fix a $3$-coloring of $G$. Now for every $v \in V(G)$ of color $c$ and for every $i \in \{0, \dots, 18g\}$, we color $v_i$ with the distance-2 color $(c + i)\mod 3$. Finally, we color the remaining non-colored vertices of $G'$ with the distance-1 color, and this coloring is a valid $\colo{1}{3}$-coloring of $G'$.
        \item If $G'$ is $\colo{1}{3}$-colorable, then for every $v \in V(G)$ and for every $i \in \{0, \dots, 18g\}$, $v_i$ is of degree $4$ so it must be colored with a distance-2 color. Moreover, for every $i$, $j \in \{0, 6g, 12g, 18g\}$, $v_i$ and $v_j$ must have the same distance-2 color.
        Therefore, for every edge $uv \in E(G)$, assuming that $u_{6i}$ is connected to $v_{6j}$ by a path of length $2$, $u_{6i}$ and $v_{6j}$ cannot be colored by the same distance-2 color.
        Thus, by assigning to every vertex $v \in V(G)$ the color of $v_0$, we obtain a proper $3$-coloring of $G$.
    \end{itemize}
    Therefore, $G$ is $3$-colorable if and only if $G'$ is $\colo{1}{3}$-colorable, so {\sc $\colo{1}{3}$-coloring} is NP-Complete even when restricted to bipartite planar graphs with maximum degree 4 and girth at least $g$ for every $g\geq 3$.
\end{proof}

\subsection{{\sc $\colo{1}{k}$-coloring} in planar graphs}

\begin{figure}
    \centering
    \includegraphics[width=0.8\linewidth]{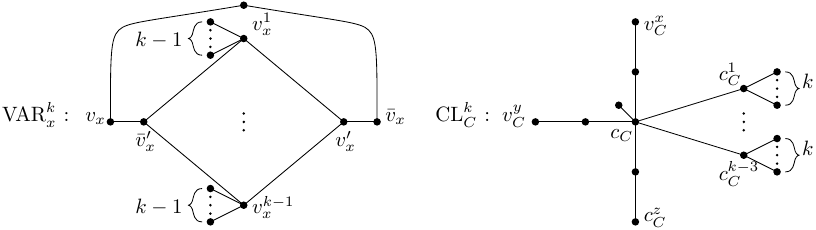}
    \caption{Two graphs used in the proof of Theorem~\ref{thm:1_k_npc}, $\textrm{VAR}_v^k$ is $\colo{1}{k}$-colorable and $v$ is colored with a distance-1 color if and only if $\bar{v}$ is colored with a distance-1 color, and $\textrm{CL}_C^k$ is $\colo{1}{k}$-colorable and at least one of $v_C^x$, $v_C^y$ or $v_C^z$ must be colored with a distance-2 color.}
    \label{fig:gadgets_for_1_k}
\end{figure}

\begin{lemma}
Let $k \geq 4$ and consider the graphs $\textrm{VAR}_x^k$ and $\textrm{CL}_C^k$ from Figure~\ref{fig:gadgets_for_1_k}. Then these two graphs are $\colo{1}{k}$-colorable, and
\begin{enumerate}
    \item in $\textrm{VAR}_x^k$, $v_x$ is colored with a distance-1 color if and only if $\bar{v}_x$ is colored with a distance-2 color, and 
    \item in $\textrm{CL}_C^k$, at least one of $v_C^x$, $v_C^y$ or $v_C^z$ must be colored with a distance-2 color.
\end{enumerate}

\begin{proof}
    It is simple to see that those two graphs are $\colo{1}{k}$-colorable.
    \begin{enumerate}
        \item Assume that in a $\colo{1}{k}$-coloring of $\textrm{VAR}_x^k$, $v_x$ and $\bar{v}_x$ are colored with the distance-1 color, then $v'_x$ and $\bar{v}_x'$ must be colored with a distance-2 color. Moreover, $v_x^1, \dots, v_x^{k - 1}$ are of degree $k + 1$ so they must be colored with a distance-2 color. Therefore, we have $k + 1$ vertices that are at distance-2 from each other and that must be colored with a distance-2 color, which is a contradiction. On the other hand, assume that in a $\colo{1}{k}$-coloring of $\textrm{VAR}_x^k$, $v$ and $\bar{v}_x$ are colored with a distance-2 color. Then they must have different distance-2 colors since they have a common neighbor. However, they both have $N(v'_x) \cup N(\bar{v}_x') - v_x - \bar{v}_x$ in their neighborhood at distance-2. So only one distance-2 color is available for both $v_x$ and $\bar{v}_x$, which is a contradiction.
        \item Assume that $\textrm{CL}_C^k$ is $\colo{1}{k}$-colorable with $v_C^x$, $v_C^y$ and $v_C^z$ using the distance-1 color. then the neighbors of $v_C^x$, $v_C^y$ and $v_C^z$ must be colored with a distance-2 color. Since $c_C$ is of degree $k + 1$, it must be colored with a distance-2 color and similarly for $c_C^1, \dots, c_C^{k - 3}$. But then $k + 1$ vertices that are at distance-2 from each other must have a distance-2 color, which is a contradiction.
    \end{enumerate}
\end{proof}
\end{lemma}

\begin{theorem}\label{thm:1_k_npc}
    For every integer $k \geq 4$, {\sc $\colo{1}{k}$-coloring} is NP-Complete even when restricted to bipartite planar graphs of maximum degree $k + 1$. 
\end{theorem}

\begin{proof}
    We reduce {\sc restricted planar $3$-sat}, which is known to be NP-Complete by Theorem~\ref{thm:NPC-problems}. Let $\varphi$ be an instance of {\sc restricted planar $3$-sat} and let $G_\varphi$ be the planar variable-clause incidence graph of $\varphi$, that is, for every variable $x$ of $\varphi$, $G_\varphi$ contains a vertex $v_x$, and for every clause $C$ of $\varphi$, $G_\varphi$ contains a vertex $v_C$. Moreover, for every clause $C$ of $\varphi$ and variable $x$ appearing in $C$, $G_\varphi$ has the edge $v_xv_C$. Let $G$ be the graph obtained from $\varphi$ and $G_\varphi$ with one copy of $\textrm{VAR}^k_{x}$ from Figure~\ref{fig:gadgets_for_1_k} per vertex $v_x \in V(G_\varphi)$, one copy of $\textrm{CL}^k_C$ from Figure~\ref{fig:gadgets_for_1_k} per vertex $v_C \in V(G_\varphi)$, and for every clause $C$ of $\varphi$ and variable $x$ appearing in $C$, if $x$ appears positively in $C$ then we identify the vertex $v_x$ from $\textrm{VAR}^k_{x}$ with the vertex $v_C^x$ from $\textrm{CL}^k_C$. Otherwise, $x$ appears negatively in $C$ in which case we identify the vertex $\bar{v}_x$ from $\textrm{VAR}^k_{x}$ with the vertex $v_C^x$ from $\textrm{CL}^k_C$. Since $\varphi$ is an instance of {\sc restricted planar $3$-sat}, every variable $x$ appears twice positively and once negatively. Observe that $G$ has maximum degree $k + 1$ and that, since $G_\varphi$ is bipartite and planar, so is $G$.
    \begin{itemize}
        \item If $\varphi$ is satisfiable, then  let us fix an assignment of truth values satisfying $\varphi$. Let $a$ be the distance-1 color and let $1, \dots, k$ be the $k$ distance-2 colors. We first color every vertex of degree $1$ in $G$ with $a$. Then, for each clause $C$ of $\varphi$, we color $c_C$ with $k - 2$, and for every $i \in \{1, \dots, k - 3\}$, we color $c^i_C$ with $i$. Then if the variable $x$ is assigned $false$, then we color $v_x$ and $v'_x$ with $a$. Otherwise, we color $\bar{v}_x$ and $\bar{v}'_x$ with $a$, and for every clause $C$ containing $x$, we color the vertex of degree $2$ that is a neighbor of $v_x$ in $\textrm{CL}_C^k$ with $a$. 
        Now consider $G[\cup_{C\in\varphi}V(\textrm{CL}_C^k)]$, we will color the uncolored vertices of this induced subgraph. Observe that for every non-colored vertex $v$ of $G[\cup_{C\in\varphi}V(\textrm{CL}_C^k)]$, every neighbor of $v$ in $G[\cup_{C\in\varphi}V(\textrm{CL}_C^k)]$ is colored with the distance-1 color. Therefore, it is enough to color $G^2$, the graph constructed from every uncolored vertices of $G[\cup_{C\in\varphi}V(\textrm{CL}_C^k)]$ such that $uv$ is an edge if $G^2$ if and only if $u$ and $v$ are at distance $2$ from each other in $G[\cup_{C\in\varphi}V(\textrm{CL}_C^k)]$. Consider $v_x$ in $G$. If $x$ was assigned $false$, the neighbors of $v_x$ that are in $G^2$ are not colored yet and only belong to even cycles in $G^2$. Indeed, every cycle in $G[\cup_{C\in\varphi}V(\textrm{CL}_C^k)]$ has a length congruent to $0$ modulo $4$ since every cycle of $G[\cup_{C\in\varphi}V(\textrm{CL}_C^k)]$ is a subdivision of a cycle of $G_\varphi$ and $G_\varphi$ is bipartite, and the vertices from $G^2$ that are the neighbors of vertices $v_x$ where $x$ is assigned $false$ are the vertices of the subdivisions.
        Therefore, we can color them with the two distance-2 colors $k - 1$ and $k$. Finally, the only uncolored vertices belong to copies of $\textrm{VAR}_x^k$. For every variable $x$ of $\varphi$, if $x$ is assigned $true$ then we color $v_x$ and $v'_x$ with the color $1$, otherwise, we color $\bar{v}_x$ and $\bar{v}_x'$ with the color $1$. Then for every $i \in \{1, \dots, k - 1\}$, we color $v_x^i$ with $i + 1$. And finally, we color the common neighbor of $v_x$ and $\bar{v}_x$ with the color $2$. One can check that this coloring is valid and therefore, we managed to $\colo{1}{k}$-color $G$ as desired.

        \item If $G$ is $\colo{1}{k}$-colorable, then let us fix a $\colo{1}{k}$-coloring of $G$. For every variable $x$ of $\varphi$, we set $x$ to be true if and only if $v_x$ is colored with a distance-2 color. Assume that this assignment of truth values does not satisfy $\varphi$. So there exists a clause $C$ containing only false literals. But then, $v_C^x$, $v_C^y$ and $v_C^z$ are colored with a distance-1 color, and so $\textrm{CL}_C^k$ is not $\colo{1}{k}$ by Lemma~\ref{lemma:tool_for_npc_gap_3_1}, which is a contradiction. 
    \end{itemize}
    So $\varphi$ is satisfiable if and only if $G$ is $\colo{1}{k}$-colorable. Since {\sc restricted planar $3$-sat} is NP-hard, so is {\sc$\colo{1}{k}$-coloring} for every $k \geq 4$, even when restricted to planar graphs of maximum degree $k + 1$.
\end{proof}

\subsection{{\sc $\colo{2}{k}$-coloring} in planar graphs}

\begin{figure}
    \centering
    \includegraphics[width=0.8\linewidth]{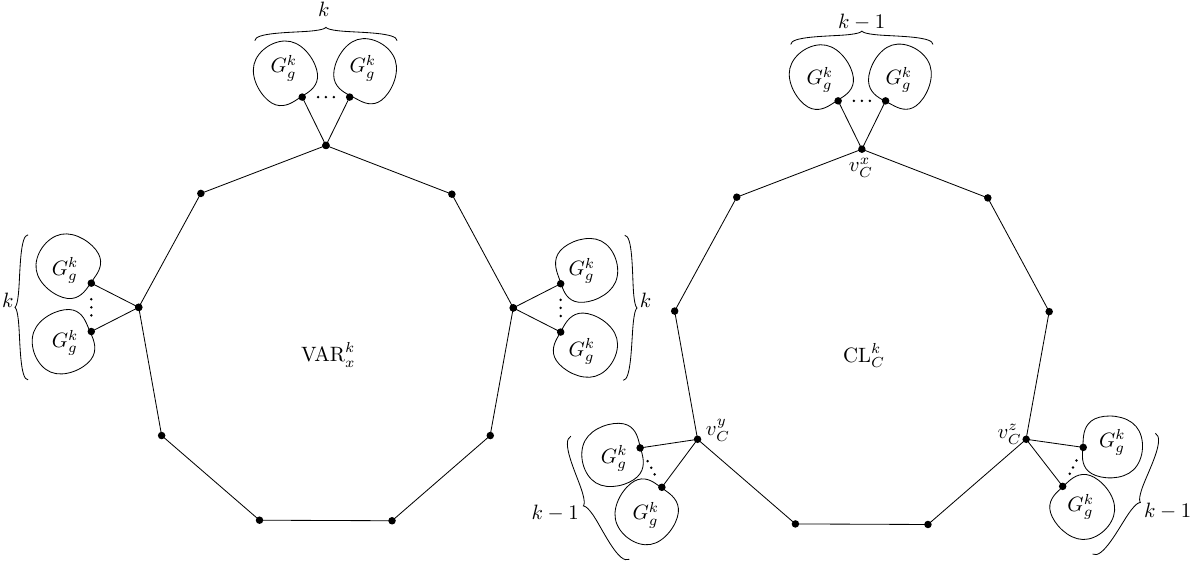}
    \caption{The variable and clause gadgets used in the proof of Theorem~\ref{thm:2k_structural_gap}.}
    \label{fig:clause_gadget_for_2_k_npc}
\end{figure}

\begin{theorem}\label{thm:2k_structural_gap}
    For every integer $k\geq 1$, $g \geq 3$, and for every class of graph $\mathcal{C}$ such that
    \begin{itemize}
        \item $\mathcal{C}$ contains $P_2$, $P_5$ and $P_8$, and
        \item every graph in $\mathcal{C}$ is planar and bipartite, and
        \item $\mathcal{C}$ is stable by disjoint union,
    \end{itemize}
    if there exists a planar graph $H_\mathcal{C}^{k, g}$ of girth at least $g$ with a vertex $s$ such that 
    \begin{itemize}
        \item in any $\colo{2}{k}$-coloring of $H_\mathcal{C}^{k, g}$, $s$ is colored with a distance-2 color, and
        \item $V(G_\mathcal{C}^{k, g})$ can be partitioned into a set of vertices inducing a graph from $\mathcal{C}$ and $k$ $2$-independent sets,
    \end{itemize}
    then it is NP-Complete to decide whether the vertex set of a planar graph $G$ of girth at least $g$ can be partitioned into a set of vertices inducing a graph from $\mathcal{C}$ and $k$ $2$-independent sets, or is not $\colo{2}{k}$-colorable.
\end{theorem}

\begin{proof}
    We reduce {\sc restricted planar 3-sat}, which is known to be NP-Complete by Theorem~\ref{thm:NPC-problems}. Let $\varphi$ be an instance of {\sc restricted planar 3-sat} and let $G_\varphi$ be the planar variable-clause incidence graph of $\varphi$, that is, for every variable $x$ of $\varphi$, $G_\varphi$ contains a vertex $v_x$, and for every clause $C$ of $\varphi$, $G_\varphi$ contains a vertex $v_C$. Moreover, for every clause $C$ of $\varphi$ and variable $x$ appearing in $C$, $G_\varphi$ has the edge $v_xv_C$.
    Let $H_\mathcal{C}^{k, g}$ be a graph such that there exists a $\colo{2}{k}$-coloring of $H_\mathcal{C}^{k, g}$ such that the vertices of $H_\mathcal{C}^{k, g}$ colored with a distance-1 color induce a graph from $\mathcal{C}$, and such that there exists a vertex $s \in V(H_\mathcal{C}^{k, g})$, for every $\colo{2}{k}$-coloring of $H_\mathcal{C}^{k, g}$, $s$ is colored with a distance-2 color. 
    We construct from $G_\varphi$ and $H_\mathcal{C}^{k, g}$ the graph $G$ as follows :
    \begin{itemize}
        \item For every variable $x$ of $\varphi$, we add a copy of the graph $\textrm{VAR}_x^k$ from Figure~\ref{fig:clause_gadget_for_2_k_npc}, and
        \item For every clause $C$ with variables $x$, $y$ and $z$, we add the gadget $\textrm{CL}_C^k$ from Figure~\ref{fig:clause_gadget_for_2_k_npc}, and
        \item For every clause containing only two literals, say $x$ and $y$, we add a copy of $\textrm{CL}_C^k$ from Figure~\ref{fig:clause_gadget_for_2_k_npc}, a copy of $H_\mathcal{C}^{k, g}$, and we connect with an edge $v_C^z$ to the vertex $s$ from the copy of $H_\mathcal{C}^{k, g}$.
        \item Finally, for every clause $C$ and every vertex $x$ appearing in $C$, if $x$ appears positively in $C$ then we connect the vertex $v_x$ to the vertex $v^C_x$ with an edge, otherwise we connect the vertex $\bar{v}_x$ to the vertex $v^C_x$ with an edge.
    \end{itemize}
    Observe that since $H_\mathcal{C}^{k, g}$, $\textrm{VAR}^k_x$ and $\textrm{CL}^k_C$ are planar, and $G_\varphi$ is the variable-clause incidence graph of $\varphi$, an instance of {\sc restricted planar 3-sat}, then $G$ is a planar graph. By Theorem~\ref{thm:choi_and_liu}, every planar graph of girth at least 10 is $\colo{2}{1}$-colorable, so we can assume that $g \leq 9$ and so $\textrm{VAR}^k_x$ and $\textrm{CL}^k_C$ are allowed to have a cycle of length $9$, so $G$ is of girth at least $g$ as well. Also, observe that both $\textrm{VAR}^k_x$ and $\textrm{CL}^k_C$ are $\colo{2}{k}$-colorable, and, in any $\colo{2}{k}$-coloring of $\textrm{VAR}^k_x$, $v_x$ is colored with a distance-1 color if and only if $\bar{v}_x$ is colored with a distance-2 color so we can $\colo{2}{k}$-color $\textrm{VAR}^k_x$ such that the vertices colored with a distance-1 color induce a $P_8$.
    Similarly, in any $\colo{2}{k}$-coloring of $\textrm{CL}^k_C$, at least one of $v_C^x$, $v_C^y$ or $v_C^z$ is colored with a distance-2 color so we can $\colo{2}{k}$-color $\textrm{CL}^k_C$ such that the vertices colored with a distance-1 color induce a disjoint union of $P_2$, $P_5$ and $P_8$. 
    \begin{itemize}
        \item If $\varphi$ is satisfiable, then let us fix a satisfiable truth assignment to the variables of $\varphi$. For every variable $x$, if $x$ is assigned true, then we color $v_x$ with a distance-1 color, and $\bar{v}_x$ with a distance-2 color. Otherwise, we color $v_x$ with a distance-2 color and $\bar{v}_x$ with a distance-1 color. Then, we can easily extend this partial coloring in order to color every copy of $\textrm{VAR}_g^k$ by coloring the uncolored vertices from the cycle of length $9$ by alternating between the two distance-1 colors, hence obtaining a $P_8$ colored with the two distance-1 colors, and then by coloring the copies of $H_\mathcal{C}^{k, g}$ such that the vertices colored with a distance-1 color induce a graph from $\mathcal{C}$. Then, since every clause is satisfied, every clause must have a literal whose value is true. If this literal is $x$, then the vertex $v_C^x$ is connected by an edge to $v_x$ which is itself already colored with a distance-1 color. Therefore, there is an available distance-2 color to color $v_C^x$. Similarly, if the literal is $\neg x$, then $v_C^x$ is connected by an edge to $\bar{v}_x$ which is itself already colored with a distance-1 color. Therefore, there is an available distance-2 color to color $v_C^x$. So in every copy of $\textrm{CL}^k_C$, at least one of $v_C^x$, $v_C^y$ or $v_C^z$ is colored with a distance-2 color, and so we can easily extend the current partial $\colo{2}{k}$-coloring in order to color every copy of $\textrm{CL}^k_C$ since the vertices from the cycle of length $9$ that are not yet colored induce a disjoint union of $P_2$, $P_5$ or $P_8$. Finally, for every clause containing only two literals, we need to color the copy of $H_\mathcal{C}^{k, g}$ that is connected from $s$ to $v_C^z$, and this is doable since one distance-2 color is available for $s$. So $G$ is $\colo{2}{k}$-colorable and the $2$ distance-$1$ colors induce a disjoint union of $P_2$, $P_5$, $P_8$ and graphs from $\mathcal{C}$.
        \item If $G$ is $\colo{2}{k}$-colorable, then let us fix a $\colo{2}{k}$-coloring of $G$. Then, for every variable $x$ of $\varphi$, we set $x$ to be true if and only if $v_x$ is colored with a distance-1 color. By a previous observation, for every clause $C$, and for any $\colo{2}{k}$-coloring of $G$, there exists a variable $x$ appearing in $C$ such that $v_C^x$ is colored with a distance-2 color. If $x$ appears positively in $C$, then $v_C^x$ is connected with an edge to $v_x$ in which case $v_x$ must be colored with a distance-1 color, and so $C$ is satisfied. Otherwise, $v_C^x$ is connected with an edge to $\bar{v}_x$, and so $\bar{v}_x$ must be colored with a distance-1 color. Therefore, $v_x$ must be colored with a distance-2 color and $x$ is assigned false. Since $x$ appears negatively in $C$, $C$ is satisfied. Therefore,every clause of $\varphi$ is satisfied, and so is $\varphi$.
    \end{itemize}
    So $\varphi$ is satisfiable if and only if $G$ is $\colo{2}{k}$-colorable, so it is NP-Complete to decide whether the vertex set of a planar graph of girth at least $g$ can be partitioned into a set of vertices inducing a graph from $\mathcal{C}$ and $k$ $2$-independent sets, or is not $\colo{2}{k}$-colorable.
\end{proof}

\begin{corollary}\label{coro:2k_gap_strutrural}
    For every integer $k \geq 1$, deciding whether the vertex set of a planar graph can be partitioned into a set of vertices inducing a disjoint union of $P_2$, $P_5$ and $P_8$ and $k$ 2-independent sets or is not $\colo{2}{k}$-colorable is NP-Complete. 
\end{corollary}

\begin{proof}

Consider a graph $G$ constructed from $k + 1$ copies of $K_3$ and a special vertex $s$ such that $s$ has been identified with exactly one vertex per $K_3$. We can $\colo{2}{1}$-color this graph by choosing any distance-2 color for $s$ and 2-coloring the remaining vertices, so the vertices colored with a distance-1 color induce a disjoint union of $P_2$. Now assume that $s$ is colored with a distance-1 color. Then each of the $k + 1$ triangles must contain a vertex different from $s$ that is colored with a distance-2 color. Since this graph is of diameter 2, we obtain a contradiction. Therefore, we can apply Theorem~\ref{thm:2k_structural_gap} with $g = 3$ and $\mathcal{C} = \{P_2, P_5, P_8\}$ to conclude.
\end{proof}

\begin{figure}
    \centering
    \begin{subfigure}{\linewidth}
        \centering
        \includegraphics[width=0.8\linewidth]{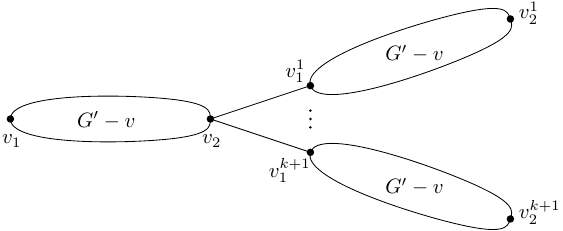}
        \caption{A graph constructed from $k + 2$ copies of $G' - v$ for the case where $\mathcal{S} = \{c, [1, \dots, k]\times[1, \dots, k]\}$ of the proof of Proposition~\ref{prop:2_k_npc_tool}. In any $\colo{2}{k}$-coloring of this graph, $v_2$ must be colored with a distance-1 color, therefore if we add a common neighbor to $v_1$ and $v_2$, we are reduced to the case where $\mathcal{S} = \{[1, \dots, k]\times[1, \dots, k]\}$.}
        \label{fig:gadget_for_2_k_tool_one}
    \end{subfigure}
    \vspace{1em}
    \begin{subfigure}{\linewidth}
        \centering
        \includegraphics[width=0.8\linewidth]{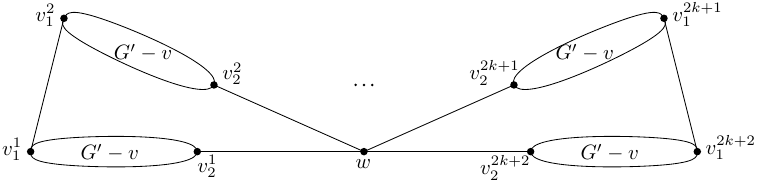}
        \caption{A graph constructed from $2k + 2$ copies of $G' - v$ for the case where  $c \in \mathcal{S}$ and there exists $(S, S') \in \mathcal{S}$ with $|S'| < k$ of the proof of Proposition~\ref{prop:2_k_npc_tool}. In any $\colo{2}{k}$-coloring of this graph, $w$ must be colored with a distance-2 color.}
        \label{fig:gadget_for_2_k_tool_two}
    \end{subfigure}
    \caption{Two graphs constructed from copies of $G' - v$ and used in the proof of Proposition~\ref{prop:2_k_npc_tool}.}
    \label{fig:gadget_for_2_k_tool}
\end{figure}

The following Proposition provides an equivalent condition for Theorem~\ref{thm:2k_structural_gap} to be applicable.

\begin{proposition}\label{prop:2_k_npc_tool}
    For every integers  $k \geq 1$ and $g \geq 3$, there exists a $\colo{2}{k}$-colorable graph $G$ of girth at least $g$ with a vertex $s$ such that in any $\colo{2}{k}$-coloring of $G$, $s$ must be colored with a distance-2 color if and only if there exists a graph $G'$ of girth at least $g$ that is not $\colo{2}{k}$-colorable.
\end{proposition}

\begin{proof}
    On one hand, let $G$ be a $\colo{2}{k}$-colorable graph of girth $g$ with a special vertex $s$ such that in any $\colo{2}{k}$-coloring of $G$, $s$ must be colored with a distance-2 color. We construct a graph $G'$ starting from $k + 1$ copies of $G$ whose special vertices are renamed $s_1, \dots, s_{k + 1}$, and then we add a vertex $v$ that is a neighbor of $s_1, \dots, s_{k + 1}$. The vertices $s_1, \dots, s_{k + 1}$ must all be colored with a distance-2 color, but they are all at distance 2 from each other hence $G'$ is not $\colo{2}{k}$-colorable. 

    On the other hand, if $g \leq 6$, then consider the graph $F_3(u; k)$ From Figure $2$ of~\cite{choi2025propersquarecoloringssparse}. This graph is of girth $6$, is $\colo{2}{k}$-colorable, and in every $\colo{2}{k}$-coloring of $F_3(u; k)$, $u$ must be colored with a distance-2 color. Otherwise, assume that $g \geq 7$, so let $G'$ be a planar graph of girth $7$ that is not $\colo{2}{k}$-colorable and that is minimal in the sense that every proper subgraph of $G'$ is $\colo{2}{k}$-colorable. Since $G'$ is of girth $7$, it is $2$-degenerate. Observe that by minimality, the minimum degree of $G'$ is $2$ so let $v$, $v_1$, $v_2 \in V(G')$ be such that $N_{G'}(v) = \{v_1, v_2\}$.
    In order, for $G'$, to not be $\colo{2}{k}$-colorable, any $\colo{2}{k}$-coloring of $G' - v$ must create one of the following obstructions that forbid any extension to a $\colo{2}{k}$-coloring of $G'$. In what follows, we say that $1, \dots, k$ are the $k$ distance-2 colors.
    \begin{itemize}
        \item $(B_c)$: the two neighbors of $v$ are colored with the same distance-2 color.
        \item $(B_{S_1, S_2})$, $S_1, S_2 \subseteq[1,\dots,k]$ with $|S_1 \cup S_2| = k$: 
        $v_1$ and $v_2$ are colored with distinct distance-1 colors, and $S_1$ (resp. $S_2$) is exactly the set of distance-$2$ colors used to color the neighbors of $v_1$ (resp. $v_2$). 
    \end{itemize}
    Let $\mathcal{S} \subseteq $\{c\}$ \cup \{(S_1, S_2) : S_1, S_2\subseteq[1,\dots,k] \textrm{ and } |S_1 \cup S_2| = k\}$. We denote by $(B_\mathcal{S})$ the event $(B_e)$ is an obstruction for the $\colo{2}{k}$-coloring of $G'$ if and only if $e \in \mathcal{S}$. We thus need to prove that $G$ exists for every possible value of $\mathcal{S}$.
    \begin{itemize}
        \item If $\mathcal{S} = \{c\}$, then in every $\colo{2}{k}$-coloring of $G' - v$, $v_1$ and $v_2$ must be colored with a distance-2 color so we can take $G = G' - v$ and $s = v_1$ or $s = v_2$.
        \item If $\mathcal{S} = \{[1, \dots, k] \times [1, \dots, k]\}$, then we construct $G$ from a copy of $G' - v$ and by adding the vertices $v_1'$ with an edge $v_1v_1'$, $v_2'$ with an edge $v_2v_2'$, and $s$ with the edges $v_1's$ and $v_2's$. It is simple to extend any $\colo{2}{k}$-coloring of $G' - v$ into a $\colo{2}{k}$-coloring of $G$. Moreover, for any $\colo{2}{k}$-coloring of $G' - v$, $v_1$ and $v_2$ are colored with distinct distance-1 colors, so in any $\colo{2}{k}$-coloring of $G$, $v_1'$ and $v_2'$ are colored with distinct distance-1 colors as well and so $s$ must be colored with a distance-2 color.
        \item If $c \not\in\mathcal{S}$, then we can assume that $\mathcal{S} - ([1, \dots, k], [1, \dots, k]) \neq \emptyset$ otherwise we are in the previous case. Assume without loss of generality that there exists $(S_1, S_2) \in \mathcal{S}$ with $|S_2| < k$. We construct the graph $G$ from two copies of $G' - v$. We then add an edge between the two copies of $v_1$ and a new vertex $s$ that is a neighbor of the two copies of $v_2$. Since $c \not\in\mathcal{S}$, both copies of $v_1$ and both copies of $v_2$ must be colored with distance-1 colors. But since the two copies of $v_1$ are neighbors, they must have distinct distance-1 colors. Therefore, the two copies of $v_2$ must have distinct distance-2 colors as well, so in any $\colo{2}{k}$-coloring of $G$, $s$ must have a distance-2 color. Observe that it is simple to have a $\colo{2}{k}$-coloring of $G - s$. Moreover, since we assumed that $|S_2| < k$, by using the same set of distance-2 colors to color the two copies of $G' - v$, at least one distance-2 color is available to color $s$, and that for any $\colo{2}{k}$-coloring of $G - s$.
        \item If $\mathcal{S} = \{c, [1, \dots, k] \times [1, \dots, k]\}$, then consider the graph from Figure~\ref{fig:gadget_for_2_k_tool_one}, constructed from $k + 2$ copies of $G' - v$ (and renaming $k + 1$ copies of $v_1$ and $v_2$ into $v_1^1, \dots, v_1^{k + 1}$ and $v_2^1, \dots, v_2^{k + 1}$) and then adding the edges $v_2v_1^i$ for every $i$ in $\{1, \dots, k + 1\}$. This graph is $\colo{2}{k}$-colorable since we can use the $\colo{2}{k}$-coloring of the event $(B_{[1, \dots, k] \times [1, \dots, k]})$ to color every copy of $G' - v$. Now, for any $\colo{2}{k}$-coloring of this graph, there must exist $i \in \{1, \dots, k + 1\}$ such that $v_1^i$ is colored with a distance-1 color. But then it must have all the distance-2 colors appearing in its neighborhood inside of its copy of $G' - v$. Therefore,$v_2$ must be colored with a distance-1 color, and $v_1$ as well. Moreover, in every  $\colo{2}{k}$-coloring of this new graph, the neighborhood of $v_1$ and $v_2$ must contain the $k$ distance-$2$ colors. Therefore, we can proceed as in the 
        $\mathcal{S} = \{[1, \dots, k] \times [1, \dots, k]\}$ case, by adding the path $v_1v_1'sv_2'v_2$ with $v_1'$, $s$ and $v_2'$ being new vertices, and $s$ must be colored with a distance-$2$ color.
        \item Finally, $c \in \mathcal{S}$ and there exists $(S_1, S_2) \in \mathcal{S}$ with either $|S_1| < k$ or $|S_2| < k$. Assume without loss of generality that there exists $(S_1, S_2) \in \mathcal{S}$ with $|S_2| < k$ and consider the graph $G$ from Figure~\ref{fig:gadget_for_2_k_tool_two} constructed from $2k + 2$ copies of $G' - v$ (in the $i$-th copy, $v_1$ and $v_2$ have been renamed $v_1^i$ and $v_2^i$) and the addition of a new vertex $s$ and the edges $sv_2^i$ for every $i \in \{1, \dots, 2k + 2\}$. Finally, for every $i \in \{1, \dots, k + 1\}$, we add an edge between $v_1^{2i- 1}$ and $v_1^{2i}$. $G$ is $\colo{2}{k}$-colorable by copying the $\colo{2}{k}$-coloring of the case $(B_{S_1, S_2})$ for every copy of $G' - v$ except that for every $i \in \{1, \dots, k + 1\}$, we have inverted the role of the two distance-1 colors in the $\colo{2}{k}$-coloring of the $(2i - 1)$-th copy of $G - v$ and in the $\colo{2}{k}$-coloring of the $(2i)$-th copy of $G - v$. Then, since we assumed that $|S_2| < k$, there must be an available distance-2 color for $s$. Now, for any $\colo{2}{k}$-coloring of $G$, there must exists $i \in \{1, \dots, k + 1\}$ such that both $v_2^{2i - 1}$ and $v_2^{2i}$ have a distance-1 color, in which case both $v_1^{2i - 1}$ and $v_1^{2i}$ have a distance-1 color as well, and distinct ones since there is an edge between $v_1^{2i - 1}$ and $v_1^{2i}$ in $G$. Therefore, $v_2^{2i - 1}$ and $v_2^{2i}$ have distinct distance-1 colors and so $s$ must be colored with a distance-2 color.
    \end{itemize}
Notice that in all the previous cases, since $G'$ is of girth at least $g$, then the distance in $G' - v$ between $v_1$ and $v_2$ is at least $g - 2$. Therefore, in each case, the constructed graph $G$ is of girth at least $g$ as well. So for every possible value of $\mathcal{S}$, there exists a graph $G$ with the desired properties.
\end{proof}

\begin{corollary}\label{coro:2k_np_hard_gorth_3456}
For every integers $k \geq 1$, $g \in  \{3, 4, 5, 6\}$, there exists a finite set $\mathcal{S}$ of bipartite planar graphs of girth at least $g$ (the size of $\mathcal{S}$ and the size of the elements of $\mathcal{S}$ only depend on $k$ and $g$) such that the problem of deciding whether the vertex set of a planar graph of girth at least $g$ can be partitioned into a set of vertices inducing a disjoint union of graphs from $\mathcal{S}$ and $k$ 2-independent sets or is not $\colo{2}{k}$-colorable is NP-Complete.
\end{corollary}

\begin{proof}
    In Figure $2$ of~\cite{choi2025propersquarecoloringssparse}, a construction of a graph of girth $6$ not $\colo{2}{k}$-colorable for every $k \geq 0$ is provided, so we can apply Proposition~\ref{prop:2_k_npc_tool} and Theorem~\ref{thm:2k_structural_gap} to conclude.
\end{proof}

As mentioned in Theorem~\ref{thm:choi_and_liu}, every planar graph of girth at least $7$ is $\colo{2}{12}$-colorable, every planar graph of girth at least $8$ is $\colo{2}{2}$-colorable, and every planar graph of girth at least $10$ is $\colo{2}{1}$-colorable. Therefore, we can provide a statement similar to the one of Corollary~\ref{coro:2k_np_hard_gorth_3456} with $g = 7$ and some $k \leq 11$ or with $k = 2$ and some $g \in \{8, 9\}$, but the precise values are not determined yet.

\subsection{{\sc $\colo{3}{1}$-coloring} in planar graphs}

\begin{figure}
    \centering
    \includegraphics[width=0.5\linewidth]{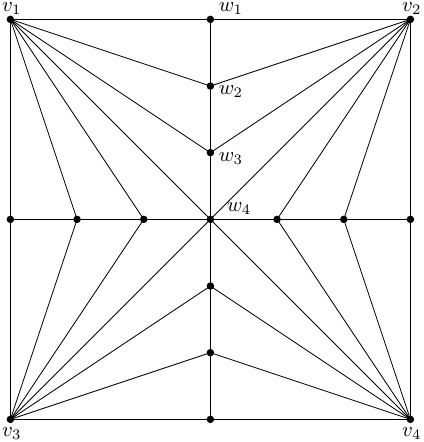}
    \caption{Vertex gadget used to represent a vertex $v$ in the proof of Theorem~\ref{thm:npc_gap_3_1}.}
    \label{fig:gadget_for_3_1_npc.pdf}
\end{figure}

\begin{lemma}\label{lemma:tool_for_npc_gap_3_1}
    Let $G$ be the graph from Figure~\ref{fig:gadget_for_3_1_npc.pdf}. Then $G$ is $\colo{3}{0}$-colorable and in every $\colo{3}{1}$-coloring of $G$, either $v_1$, $v_2$, $v_3$ and $v_4$ have the same distance-1 color or three of them have the same distance-1 color and the forth one has the distance-2 vertex.
\end{lemma}

\begin{proof}
    Observe that among $v_1$, $v_2$, $v_3$ and $v_4$, at most one vertex can have the distance-2 color. Assume that two consecutive corners of $G$ are colored with different distance-1 colors, then we can assume without loss of generality that $v_1$ and $v_2$ have different distance-1 colors. Then at least one of $w_1w_2$, $w_2w_3$ or $w_3w_4$ is such that both endpoints have distance-1 colors, in which case there is a monochromatic edge, so this case cannot happen in a $\colo{3}{k}$-coloring of the graph from Figure~\ref{fig:gadget_for_3_1_npc.pdf}. The remaining case is if two opposite corners of $G$ have distinct distance-1 colors, in which case we can assume that $v_1$ and $v_3$ have distinct distance-1 colors. Then at most one of $v_2$ and $v_4$ is colored with the distance-$2$ color, and therefore at least one of $v_2$ or $v_4$ is colored with a distance-$1$ color distinct from the distance-1 color of $v_1$ or $v_3$ thus we must be in the previous case.
\end{proof}

\begin{theorem}\label{thm:npc_gap_3_1}
Given a planar graph $G$ that is either 3-colorable or not $\colo{3}{1}$-colorable,
it is NP-Complete to decide whether $G$ is 3-colorable.
\end{theorem}
\begin{proof}
We reduce {\sc 3-coloring}, which is known to be NP-Complete on planar graphs with maximum degree 4 by Theorem~\ref{thm:NPC-problems}. Let $G$ be a planar graph of maximum degree $4$. We construct from $G$ the graph $G'$ with, for every vertex $v \in V(G)$, a copy of the graph from Figure~\ref{fig:gadget_for_3_1_npc.pdf} whose corner vertices are labeled $v_1$, $v_2$, $v_3$ and $v_4$, and, for every edge $uv \in E(G)$, we chose one $v_iv_{i + 1}$ face from the vertex gadget of $v$, one $u_ju_{j + 1}$ face from the vertex gadget of $u$, and we add the edges $v_iu_{j + 1}$, $v_{i + 1}u_j$ and $v_iu_j$ to $G'$. Observe that since $G$ is planar, for each step of the procedure, it is possible to choose $i$ and $j$ such that $G'$ is planar as well.
\begin{itemize}
    \item If $G$ is $3$-colorable, then for each vertex $v \in V(G)$, we color $v_1$, $v_2$, $v_3$ and $v_4$ with the color of $v$. The remaining graph can be uniquely colored with only distance-1 colors. 
    \item On the other hand, if $G'$ is $\colo{3}{1}$-colorable, then by Lemma~\ref{lemma:tool_for_npc_gap_3_1}, for every vertex $v \in V(G)$, at most one of $v_1$, $v_2$, $v_3$ and $v_4$ has the distance-2 color and the other ones have the same distance-1 color. Moreover, for every edge $uv \in E(G)$, let $i$, $j$ be such that $v_iu_{j + 1}$, $v_{i + 1}u_j$ and $v_iu_j$ are edges in $G'$. Then at most one of $v_i$, $v_{i + 1}$, $u_j$ and $u_{j + 1}$ is colored with the distance-2 color. Therefore, the distance-1 color used to color the corner vertices of the vertex gadget of $v$ must be different than the distance-1 color used to color the corner vertices of the vertex gadget of $u$. Thus, copying these colors leads to a proper $3$-coloring of $G$. 
\end{itemize}
Finally, if $G'$ is $\colo{3}{1}$-colorable, then one can observe that every vertex of $G'$ cannot see $3$ different distance-1 colors in its neighborhood, so every vertex colored with a distance-2 color can be recolored with a distance-1 color, so $G'$ is $\colo{3}{1}$-colorable if and only if it is $3$-colorable.
\end{proof}

\subsection{{\sc$\colo{3}{k}$-coloring} in planar graphs}

\begin{figure}
    \centering
    \includegraphics[width=0.5\linewidth]{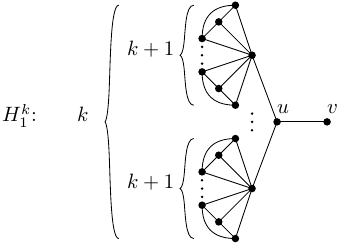}
    \caption{The gadgets used in the proof of Theorem~\ref{thm:npc_3_k}. In $H_1^k$ both $u$ and $v$ must be colored with distance-1 colors.}
    \label{fig:gadgets_for_3_k_npc}
\end{figure}

\begin{theorem}\label{thm:npc_3_k}
    For every integer $k\geq 1$, {\sc$\colo{3}{k}$-coloring} is NP-Complete even when restricted to planar graphs of maximum degree $3k + 4$.
\end{theorem}

\begin{proof}
    We reduce {\sc $3$-coloring}, which is known to be NP-Complete even when restricted to planar graphs of maximum degree $4$ by Theorem~\ref{thm:NPC-problems}.
    Consider the graph $H_1^k$ from Figure~\ref{fig:gadgets_for_3_k_npc}, then in any $\colo{3}{k}$-coloring, both $u$ and $v$ must be colored with a distance-1 color. Indeed, one can observe that any vertex in $N(u) - v$ must be colored with a distance-2 color and $|N(u) - v| = k$, so no distance-2 color is available for $u$ and $v$. 
    Let $G$ be a planar graph of maximum degree $4$. We construct the graph $G'$ starting from a copy of $G$ and, for every vertex $w \in V(G)$, we add a copy of $H_1^k$, and we identify $w$ and $v$. Notice that $G'$ is a planar graph of maximum degree $3k + 4$.
    \begin{itemize}
        \item If $G$ is $3$-colorable, then for every vertex $v \in V(G)$, we color $v$ in $G'$ with the same distance-1 color. By the previous observations, this pre-coloring of $G'$ can be extended into a $\colo{3}{k}$-coloring.
        \item  If $G'$ is $\colo{3}{k}$-colorable, then by the previous observations, every vertex in $G'$ corresponding to a vertex in $G$ was colored with a distance-1 color. Therefore, the $\colo{3}{k}$-coloring of $G'$ induces a $3$-coloring of $G$.
    \end{itemize}
    So $G$ is $3$-colorable if and only if $G'$ is $\colo{3}{k}$-colorable, hence {\sc $\colo{3}{k}$-coloring} is NP-Complete even when restricted to planar graphs of maximum degree $3k + 4$.
\end{proof}

Notice that here, we do not need to consider planar graphs of large girth due to the following celebrated Grötszch's Theorem.
\begin{theorem}[Grötszch \cite{grunbaum1963grotzsch, thomassen2003short}, 1959]\label{thm:grotzcsh_thm}
    Every triangle-free planar graph is $3$-colorable.
\end{theorem}

\section{Extremal values for the required number of distance-2 colors}\label{sec:extr_values}

In this section, we are interested, for a fixed value of $a$, in the maximum value of $b$ required to $\colo{a}{b}$-color $(a + 1)$-colorable graphs of order $n$ with some structural constraints. More precisely, let $\mathcal{C}$ be any class of graphs. We denote by  $\ex{k}{\mathcal{C}}{}$ the function that associates to every $n \in \mathbb{N}$ the smallest integer $c$ such that all the graphs of order $n$ in $\mathcal{C}$ are $\colo{k}{c}$-colorable and we therefore are interested in determining the value of $\ex{a}{\mathcal{C}}{(n)}$ when $\mathcal{C}$ is a $(a + 1)$-colorable class of graphs. For instance, Theorem~\ref{thm:choi_and_liu} states that if $\mathcal{P}_g$ is the class of planar graphs of girth $g$ then for every $n \in \mathbb{N}$, $\ex{2}{\mathcal{P}_7}{(n)} \leq 12$, $\ex{2}{\mathcal{P}_8}{(n)} \leq 2$ and $\ex{2}{\mathcal{P}_{10}}{(n)} = 1$. 

\subsection{Extremal values for $k$-degenerate graphs}

One of the simplest classes of $(k + 1)$-colorable graphs is the class of $k$-degenerate graphs. In this section, we will see that every $k$-degenerate graph is $\colo{k}{O(\sqrt{n})}$-colorable. Most of the upper-bounds on the required number of distance-2 colors are obtained by following the idea of the next remark.

\begin{remark}\label{rem:colo_strat}
    Let $G$ be a graph and let $S \subseteq V(G)$. If $\chi_2(G - S) \le c$ and $G[N_G[S]]$ is $\colo{a}{b}$-colorable such that no two vertices of $N_G(S)$ have the same distance-2 color, then $G$ is $\colo{a}{b + c}$-colorable.
\end{remark}

Indeed, one can first $2$-distance $c$-color $G - S$ and then $\colo{a}{b}$-color $G[N_G[S]]$ such that no two vertices of $N_G(S)$ have the same distance-2 color. Doing so,  
since the $2$-distance $c$-coloring of $G - S$ is valid, it remains valid in $G - N_G[S]$. Moreover, each distance-$2$ color used in $N_G(S)$ appears once, so no conflict is created, and therefore the coloring is valid.

\begin{lemma}\label{lemma:k-degenerate-dominated-colo-bound}
     Let $G$ be a $k$-degenerate graph with a dominating set $S$ of size $N$. Then there exists a set $T$ containing $S$ and of size at most $(k + 1)\times N$ such that $G - T$ is $k$-colorable.
\end{lemma}

 \begin{proof}
    We will $k$-color $G - T$ and construct $T$ at the same time, following a $k$-degeneracy ordering on $V(G)$. So let $v_1, \dots, v_n$ be an ordering of the vertices of $G$ and let $G_i = G[\{v_1, \dots, v_i\}]$ be such that for all $i \geq 1$, $v_i$ has degree at most $k$ in $G_i$. If $i = 1$, then if $v_1 \in S$ then we add $v_1$ to $T$. Otherwise, every color is available for $v_1$, so $G_1$ is $k$-colorable. If $i \geq 2$, then if $v_i \in S$ then we add it to~$T$. Otherwise, if among the at most $k$ neighbors of $v_i$ in $G_i$ there exists a vertex in $T$, then at most $k - 1$ neighbors of $v_i$ are colored, thus at least one color is available for $v_i$, and we color $v_i$ with this color. Finally, if no neighbors of $v_i$ in $G_i$ is in $T$, then we add $v_i$ to $T$. At the end of the procedure, $S \subseteq T$. Moreover, $|T - S| \leq k\times N$. Indeed, since $S$ is a dominating set of $G$, every vertex $v_i$ of $T - S$ has been added to $T$ because for every $v_j \in S$ dominating $v_i$, we have $i < j$ in which case we say that $v_j$ is \emph{late} for $v_i$. Since $G$ is $k$-degenerate, every vertex of $S$ can be late for at most $k$ vertices therefore $|T - S| \leq k\times|S| = k\times N$, so $|T| \leq (k + 1)\times N$ as desired.
\end{proof}

If a $k$-degenerate graph $G$ has a dominating set of size $N$, then we can construct $T$ from Lemma \ref{lemma:k-degenerate-dominated-colo-bound}, color each vertex in $T$ with a unique distance-2 color, and then $k$-color $G - T$. This leads to the following Corollary of Lemma \ref{lemma:k-degenerate-dominated-colo-bound}.

\begin{corollary}\label{coro:k-degenerate-dominated-colo-bound}
    For every $k$-degenerate graph $G$ of order $n$ with a dominating set $S$ of size $N$, $G$ is $\colo{k}{(k + 1)\times N}$-colorable. Moreover, it is possible to use every distance-2 color at most once each.
\end{corollary}

We further require bounds on the $2$-distance chromatic number of $k$-degenerate graphs. The following Theorem provides an upper-bound linear in the maximum degree.

\begin{theorem}[Krumke, Marathe and Ravi \cite{krumke2001models}, 2001]\label{thm:dist-2_colo_of_k-dege}
    Let $G$ be a $k$-degenerate graph of maximum degree $\Delta$, then $\chi_2(G) \le (2k - 1)\Delta + 1$.
\end{theorem}

We can thus use the idea from Remark \ref{rem:colo_strat} together with Corollary \ref{coro:k-degenerate-dominated-colo-bound} and Theorem \ref{thm:dist-2_colo_of_k-dege} to obtain the following upper-bound.

\begin{theorem}\label{thm:coloring_k_degenerate_graphs}
   Let $\mathcal{D}_k$ be the class of $k$-degenerate graphs. Then $\ex{k}{\mathcal{D}_k}{(n)} \leq 4k\sqrt{k + 1}\sqrt{n}$.
\end{theorem}

\begin{proof}
    Let $G$ be a $k$-degenerate graph of order $n$ and let $S$ be the set of vertices of degree at least $N$ in $G$, $N$ to be determined later. In particular, since $G$ is $k$-degenerate, $G$ has at most $kn$ edges so, in particular, $|S| \le 2kn/N$. Since $S$ is a dominating set of $N_G[S]$, by Corollary \ref{coro:k-degenerate-dominated-colo-bound}, $N_G[S]$ is $\colo{k}{(k + 1)2kn/N}$-colorable and it is possible to use each distance-2 colors at most once. Moreover, $G - S$ is of maximum degree $N - 1$ and therefore, $\chi_2(G - S) \le (2k - 1)(N - 1) + 1$ by Theorem \ref{thm:dist-2_colo_of_k-dege} and $(2k - 1)(N - 1) + 1 \le 2kN$. Therefore, by Remark \ref{rem:colo_strat}, $G$ is $\colo{k}{2kN + (k + 1)2kn/N}$-colorable. By taking $N = \sqrt{k + 1}\sqrt{n}$, we use at most $4k\sqrt{k + 1}\sqrt{n}$ distance-2 colors as desired.
\end{proof}

Regarding lower bounds on $\ex{k}{\mathcal{D}_k}{(n)}$, consider the graph $G$ from Figure~\ref{fig:k-degenerate graph not 2k-colorable}. It is a partial $k$-tree and is of order $kl^2 + (1 + 2k)l + k + 2$. Moreover, it is not $\colo{k}{l}$-colorable. Indeed, since $c$ is of degree $l + 1$ in $G$, one of $u_1, \dots, u_{l + 1}$ must be colored with a distance-1 color, say $u_i$. Therefore, at least one vertex per clique of size $k + 1$ that is connected to $u_i$ must be colored with a distance-2 color, which is a contradiction. This provides the lower-bound \[\ex{k}{\mathcal{D}_k}{(n)} \ge \sqrt{n - 1}/\sqrt{k}.\]

It is well-known that the graphs of tree-width $2$ are planar graphs. More precisely, the graphs of tree-width $2$ are the graphs whose 2-connected components are subgraphs of series-parallel graphs. This class captures, in particular, the outerplanar graphs. Since graphs of tree-width $2$ are $2$-degenerate, Theorem~\ref{thm:coloring_k_degenerate_graphs} provides the upper-bound $\ex{2}{\mathcal{T}_2}{(n)} \le 8\sqrt{3}\sqrt{n}$ where $\mathcal{T}_2$ is the class of graphs of tree-width at most $2$. This is tight up to a constant factor, even when restricted to the class of cactus graphs, the graphs whose 2-connected components are cycles or $K_2$ since the graph from Figure~\ref{fig:k-degenerate graph not 2k-colorable} is a cactus graph when $k = 2$. Therefore, if $\mathcal{O}_1$ is the class of outerplanar graphs (so $\mathcal{O}_1 \subset \mathcal{T}_2$), then $\ex{2}{\mathcal{O}_1}{(n)}\geq \sqrt{n}/\sqrt2$. In the construction of Figure~\ref{fig:k-degenerate graph not 2k-colorable} however, for $k = 2$, the graph contains many triangles. In fact, a cactus-graph must contain triangles to require many distance-2 colors.
\begin{figure}
    \centering
    \includegraphics[width=0.8\linewidth]{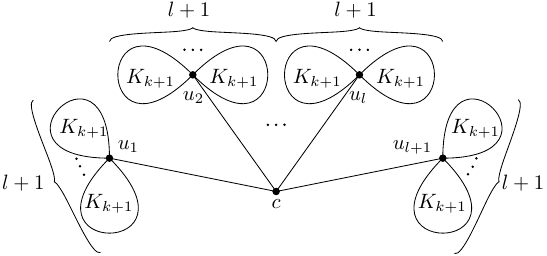}
    \caption{A partial $k$-tree of order $kl^2 + (1 + 2k)l + k + 2$ not $\colo{k}{l}$-colorable.}
    \label{fig:k-degenerate graph not 2k-colorable}
\end{figure}

\begin{proposition}
    Every cactus-graphs of girth at least $4$ is $\colo{2}{1}$-colorable.
\end{proposition}

\begin{proof}
    Let $G$ be a cactus graph of girth at least $4$ and of order $n$, and assume without loss of generality that $G$ is connected. We proceed by induction on the number of 2-connected components of the graph. If $G$ has just one connected component, then it is just a cycle and therefore $G$ is $\colo{2}{1}$-colorable. Otherwise, consider $C$, a 2-connected component of $G$ containing exactly one cut-vertex ($C$ is called a leaf block). Let $u$ be the cut-vertex that belongs to $C$. If $C$ induces a $K_2$ in $G$, then let $v$ be the vertex in $C - u$. Since $v$ has degree $1$, we can $\colo{2}{1}$-color $G - v$ by induction and then at least one distance-1 color will be available for $v$. Otherwise, let $v$ be a vertex at distance at least $2$ from $u$ in $C$. Since $G$ has girth $4$, $C$ is a cycle of length at least $4$ so $v$ exists. We $\colo{2}{1}$-color $G - C+u$ by induction. If $u$ is colored with a distance-2 color, then we can color $C - u$ only with distance-1 colors. Otherwise, we can color $v$ with the distance-2 color and $C - u - v$ with distance-1 colors.
\end{proof}

\subsection{Extremal values for planar graphs}

In this section, we study the maximum number of distance-2 colors required to color a $(a + 1)$-colorable planar graph when $a$ distance-1 colors are available.  

\subsubsection{Bounds on $\ex{1}{\mathcal{G}}{}$}

In the case $a = 1$, the graph from Figure~\ref{fig:bound_bipartitr_planar_linear_1_k} is of order $4k$ and is not $\colo{1}{k}$-colorable. Indeed, assume otherwise that this graph is $\colo{1}{k}$-colorable. Since $\deg(s) = \deg(t) > k$, $s$ and $t$ must be colored with a distance-2 color. So at least $k$ neighbors of $s$ must be colored with a distance-1 color, and similarly for $t$, so there exists $i$ such that the edge $u_iv_i$ is monochromatic, which is a contradiction. Notice that the graph from Figure~\ref{fig:bound_bipartitr_planar_linear_1_k} is $2$-outer planar, bipartite, and has tree-width $2$. So this construction provides the lower bound $\ex{1}{\mathcal{C}}{(n)} \geq n/4$ even when $\mathcal{C}$ is the class of $2$-outer planar, bipartite graphs of tree-width at most $2$. Here, $\mathcal{C}$ is minimal in the sense that if we consider subclasses of $\mathcal{C}$, we quickly obtain graphs that are $\colo{1}{o(n)}$-colorable.

\begin{figure}
    \centering
    \includegraphics[width=0.5\linewidth]{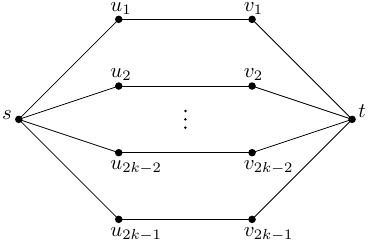}
    \caption{A bipartite planar graph of girth $6$ and of order $4k$ not $\colo{1}{k}$-colorable.}
    \label{fig:bound_bipartitr_planar_linear_1_k}
\end{figure}

In order to upper-bound the number of distance-2 colors necessary to color a triangle-free outerplanar graph when only one distance-1 color is available, we first need to upper-bound the size of a minimum vertex cover of a triangle-free outerplanar graph. To do so, we will prove that any such graph is homomorphic to $C_5$. A \emph{graph homomorphism} from $G$ to $H$ is a function $f : V(G) \rightarrow V(H)$ such that if $uv \in E(G)$ then $f(u)f(v) \in E(H)$ in which case we say that $G$ is \emph{homomorphic} to $H$.

\begin{lemma}[Pinlou and Sopena \cite{pinlou2006oriented}, 2006]\label{lemma:pinlou2006oriented}
    Every outerplanar graph $G$ of girth $g$ and of minimum degree at least $2$ contains a face of size $l \ge g$ with at least $l - 2$ consecutive vertices of degree $2$.
\end{lemma}

\begin{proposition}\label{prop:homomorphism}
    Let $g = 2k$ or $g = 2k + 1$. Every outerplanar graph of girth at least $g$ is homomorphic to $C_{2k + 1}$.
\end{proposition}

\begin{proof}
    Let $k$ be such that $g = 2k$ of $g = 2k + 1$. Assume for the sake of contradiction that the previous statement is false, and let $G$ be a minimal counter-example. If $G$ contains a cut-vertex $v$, then let $H_1, \dots, H_c$ be the connected components of $G - v$. Then $H_1 + v, \dots, H_c + v$ are all homomorphic to $C_{2k + 1}$ by minimality of $G$, and therefore $G$ is homomorphic to $C_{2k + 1}$, which is a contradiction. So we can assume that $G$ is $2$-connected and therefore has minimum degree at least $2$ (or $G$ is $P_1$ or $P_2$ in which case $G$ is homomorphic to $C_{2k + 1}$, which is a contradiction). Then by Lemma \ref{lemma:pinlou2006oriented}, $G$ contains a face $F$ of size $l \ge g$ with at least $l - 2$ consecutive vertices of degree $2$. Let $P$ be those consecutive vertices and let $u$, $v$ be the vertices of $F - P$. By minimality of $G$, $G - P$ is homomorphic to $C_{2k + 1}$ so let $f$ be a homomorphism from $G - P$ to $C_{2k + 1}$. Since $uv \in E(G)$, $f(u) \neq f(v)$ so there exists $P_o$, $P_e$, the two paths from $f(u)$ to $f(v)$ in $C_{2k + 1}$ such that $P_o$ is of odd length and $P_e$ is of even length. If $P$ is of odd length, then $F$ is homomorphic to $P_o$ and if $P$ is of even length, then $F$ is homomorphic to $P_e$. Therefore, $G$ is homomorphic to $C_{2k + 1}$, which is a contradiction, so a minimal counter-example does not exist and the Proposition follows.
\end{proof}

\begin{corollary}\label{coro:vertex_cover_coro}
    Let $G$ be an outer-planar graph of girth $g$ and of order $n$, then $G$ has a vertex cover of order at most $n(k + 1)/(2k + 1)$ where $k$ is such that $g = 2k$ or $g = 2k + 1$.
\end{corollary}

\begin{proof}
    By Proposition \ref{prop:homomorphism}, $G$ is homomorphic to $C_{2k + 1}$. For every independent set $S$ of $C_{2k + 1}$, $f^{-1}(S)$ is an independent set of $G$. Let $\mathcal{S}_k$ be the set of independent sets of size $k$ of $C_{2k + 1}$, we have
    \[\sum_{S \in \mathcal{S}_k}|f^{-1}(S)| = \sum_{v \in V(G)} |\{S \in \mathcal{S}_k: v \in S\}| = nk\]
    since every vertex of $C_{2k + 1}$ belongs to exactly $k$ maximum independent sets of $C_{2k + 1}$. Moreover, $|\mathcal{S}_k| = 2k + 1$ so by the pigeon-hole principle, there exists $S \in \mathcal{S}_k$ such that $|f^{-1}(S)| \ge nk/(2k + 1)$ and therefore $G$ has a vertex cover of size at most $n(k + 1)/(2k + 1)$.  
\end{proof}

Notice that Corollary \ref{coro:vertex_cover_coro} is tight since the smallest vertex cover of a disjoint union of $n$ $C_{2k + 1}$ is exactly $(k + 1)n$. We are now ready to prove the following proposition that bounds the required number of distance-2 color in forests and in triangle-free outerplanar graphs when one distance-1 color is available.

\begin{proposition}\label{prop:upperbound_colo_for_forests_and_triangle-free_outerplanar_graphs}
    Forests are $\colo{1}{4\sqrt{2}\sqrt{n}}$-colorable and triangle-free outerplanar graphs are $\colo{1}{4\sqrt{34/5}\sqrt{n} - 1}$-colorable.
\end{proposition}

\begin{proof}
    Forests are $1$-degenerate so we can apply Theorem~\ref{thm:coloring_k_degenerate_graphs} to $\colo{1}{4\sqrt{2}\sqrt{n}}$-color every forests. 
    Let $G$ be a triangle-free outerplanar graph, and let $S$ be the set of vertices of degree at least $\alpha \sqrt{n}$ for some real number $\alpha$ to be optimized later. Since $G$ is outerplanar, it is $2$-degenerate, so $G$ has at most $2n$ edges and therefore $|S| \leq 4\sqrt{n}/\alpha$. 

    In order to color $G$, our first goal is to construct a set $S' \subseteq N_G[S]$ such that $S \subseteq S'$, $|S'| = \Theta(|S|)$ and $N_G[S] - S'$ is an independent set. If $|S| = 1$, then since $G$ is triangle-free, $N_G[S] - S'$ is an independent set. Otherwise, we claim that if $|S| \geq 2$, then $|N_G[S]| \le 5|S|-4$.
    Indeed, we proceed by induction on $\kappa_2$, the number of $2$-connected components of $G[N_G[S]]$. If $\kappa_2 = 1$, then the outerface of $G[N_G[S]]$ is a cycle and since $S$ is a dominating set of $G[N_G[S]]$ and $|S| \ge 2$, we have $|N_G[S]| \le 3|S| \le  5|S| - 4$. Now assume that $\kappa_2 \geq 2$ and let $B$ be a leaf block of $G[N_G[S]]$, that is a $2$-connected component of $G[N_G[S]]$ with exactly one cut-vertex $v$. Let $m = |S \cap B|$, we then have, 
        \[|G[N_G[S]]| = |G[N_G[S] - (B - v)]| + |G[B]| - 1\]
    since $G[N_G[S] - (B - v)]$ and $G[B]$ only interest on $v$. So by induction hypothesis, 
        \[|G[N_G[S]]| \le 5(|S| - (m - 1)) - 4 + 5m - 4 - 1 = 5|S| - 4\]
    as desired. Now consider the set 
    \[V' = \{v | \exists svv's' \textrm{ a path of } G \textrm{ with } s, s' \in S \textrm{ and } v, v' \in N_G[S] - S\}\]
    Then by the previous claim, and since $V' \subseteq N_G[S] - S$, $|V'| \le 4|S| - 4 \le 4|S|$ so, in particular, there exists a set $C \subseteq V'$ that is a vertex cover of $G[V']$ of size at most $12|S|/5$ since by Corollary \ref{coro:vertex_cover_coro}, every outerplanar graph or order $n$ has a vertex cover of size at most $3n/5$. We thus define $S'$ to be $S \cup C$ so $|S'| \le 17|S|/5$. Assume that $N_G[S] - S'$ is not an independent set. Then there exists $u$, $v \in N_G[S] - S'$ such that $uv$ is an edge of $G[N_G[S]]$. Since $S$ is a dominating set of $N_G[S]$ and $G$ is triangle-free, there must exist $s$, $s' \in S$ such that $suvs'$ is a path in $G[N_G[S]]$. But $u$ and $v$ belong to $V'$ and $uv$ is not covered by $C$ and this is a contradiction.

    We can now color $G$. We first construct the $2$-tree $T$ that is a supergraph of $G$ and that we will color. Since $G$ is outerplanar, it has treewidth at most $2$ so $T$ exists. We first color every vertex from $S'$ with a unique distance-2 color each, at the cost of at most $|S'| \leq 17|S|/5 \le 68\sqrt{n}/(5\alpha)$ distance-2 colors. Then, we color every vertex in $N_G(S) - S'$ with a distance-1 color. This pre-coloring of $T$ might not be valid, but it is valid when we reduce it to $G$ since as previously observed, $N_G(S) - S'$ is an independent set in $G$ and each distance-2 color is used once. Then, we color the remaining vertices of $T$ following a $2$-degenerate ordering provided by the construction of this $2$-tree. 
    Let $v_1, \dots, v_n$ be an ordering of the vertices and $T_i$ be the subgraph of $T$ induced by $\{v_1, \dots, v_i\}$ such that for every $i$, $\deg_{T_i}(v_i) = 2$. By definition of a $2$-tree, the two neighbors of $v_i$ are neighbors themselves. Let $A$ be a set of distance-2 colors distinct from the distance-2 colors used to color vertices in $S'$ whose size will be determined later. If $v_1 \not\in N_G[S]$ was precolored, then we do not recolor it. Otherwise, we can color it with any distance-2 color from $A$. 
    Then, assume that $i \ge 2$ and that we managed to color $T_{i - 1}$ while respecting the precoloring. We thus need to color $v_i$ in order to extend our coloring up to $T_i$.
    As previously observed, the two neighbors of $v_i$ in $G_i$ are neighbors themselves so they do not have the same color. Therefore, by adding $v_i$ to $T_{i - 1}$, we do not create conflicts. If $v_i$ was precolored, then we do not recolor it. Otherwise, the neighbors of $v_i$ are of degree at most $\alpha \sqrt{n} - 1$ in $G_{i - 1}$ and so at most $2\alpha \sqrt{n} - 2$ distance-2 colors are unavailable for $v_i$. Therefore, if $|A| \geq 2\alpha \sqrt{n} - 1$, we can color $T$ such that there is no conflict involving the distance-2 colors. When reducing this coloring to $G$, there is still no conflict involving distance-2 colors since $G$ is a subgraph of $T$ and there is no conflict involving distance-1 colors since the vertices colored with the distance-1 color induce an independent set. In total, we used at most $|S'| + |A| \le 68\sqrt{n}/(5\alpha) +  2\alpha \sqrt{n} - 1$ distance-2 colors and this quantity is minimized for $\alpha = \sqrt{34/5}$ in which case we managed to $\colo{1}{4\sqrt{34/5}\sqrt{n} - 1}$-color $G$ as desired.
\end{proof}

Proposition \ref{prop:upperbound_colo_for_forests_and_triangle-free_outerplanar_graphs} is tight up to a constant factor since, for instance, the graph the graph $G$ from Figure \ref{fig:k-degenerate graph not 2k-colorable} with $k = 1$ is a tree of order $1 + (l + 1)^2$ and is not $\colo{1}{l}$-colorable. This provides the lower-bound $\ex{1}{\mathcal{C}}{(n)} \ge \sqrt{n - 1}$ whether $\mathcal{C}$ is the class of trees or the class of triangle-free outerplanar graphs.

\subsubsection{Bounds on $\ex{2}{\mathcal{G}}{}$} 

\begin{figure}
    \centering
    \includegraphics[width=0.5\linewidth]{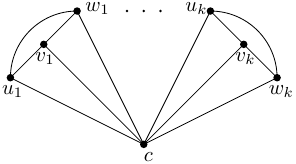}
    \caption{A 2-outerplanar graph of order $3k + 1$ not $\colo{2}{k}$-colorable.}
    \label{fig:planar_linear_bound}
\end{figure}

If we consider the class $\mathcal{P}$ of planar graphs, then Figure~\ref{fig:planar_linear_bound} provides a linear lower bound on $\ex{2}{\mathcal{P}}{}$. 
Indeed, let $G$ be the graph of Figure~\ref{fig:planar_linear_bound} and assume that $G$ is $\colo{2}{k}$-colorable. If $c$ is colored with a distance-2 color, then since each triangle in $G - c$ must contain a vertex that is colored with a distance-2 color, at least $k + 1$ vertices are colored with a distance-2 color. $G$ being of diameter $2$, we are in a state of contradiction. Otherwise, $c$ is colored with a distance-1 color, but then in every triangle of $G - c$, at least two vertices must be colored with a distance-2 color. So in total, at least $2k$ vertices must be colored with a distance-2 color, and since $G$ is of diameter $2$ and $k \ge 1$, we found a contradiction as well. Since $G$ is of order $3k + 1$ and is not $\colo{2}{k}$-colorable, we obtain the lower-bound $\ex{2}{\mathcal{P}}{(n)}\geq (n - 1)/3$. 

It is classical in the context of graph coloring problems or graph partition problems to study triangle-free graphs or graphs of higher girth. Let $\mathcal{P}_g$ be the class of planar graphs of girth at least $g$, as we will see in what follows,
$\ex{2}{\mathcal{P}_4}{}$ is sublinear. In order to color such graphs with $O(\sqrt{n})$ distance-2 colors, we will need the following two lemmas.

For a graph $G$, an \emph{odd cycle transversal} of $G$ is a set $S \subseteq V(G)$ such that for every cycle $C$ of $G$ of odd length, $C \cap S \neq \emptyset$. We denote by $\oct(G)$ the smallest integer $k$ such that there exists an odd cycle transversal of $G$ of size $k$. Notice that for every odd cycle transversal $S$ of $G$, $G - S$ is $2$-colorable.

\begin{lemma}\label{lemma:oct_bound}
    Let $G$ be a triangle-free planar graph and let $D$ be a connected dominating set of $G$. Then $\oct(G - D) \le 5|D|/3$.
\end{lemma}

\begin{proof}
    Let $C$ be a $2$-connected component of $G - D$. Since $D$ is a dominating set, every vertex in $C$ must have a neighbor in $D$. Moreover, since $D$ is connected, the vertices of $D$ dominating $C$ must belong to the same face $f$ of $G[D]$. So let $P$ be the smallest path in $G[D]$ along $f$ containing all the vertices from $D$ dominating $C$. Then we say that the edges of $P$ are \emph{jailed} by $C$ and $C$ jails at least $|C| - 1$ edges. By planarity of $G$, each edge of $G[D]$ can be jailed at most twice. Let $\mathcal{B}$ be the set of maximal $2$-connected components of $G - D$ of size at least $5$. We thus have
    \[4|\mathcal{B| \le }\sum_{C \in \mathcal{B}} |C| - 1 \le 2|\{e\in E(G[D]), e \textrm{ is jailed}\}| \le 2|E(G[D])|\]
    so, in particular, 
    \[\sum_{C \in \mathcal{B}} |C| = |\mathcal{B}| + \sum_{C \in \mathcal{B}} |C| - 1 \le  5|E(G[D])|/2\]
    Therefore, the number of vertices of $G - D$ that belong to a $2$-connected component of $G - D$ containing an odd cycle is at most $5|E(G[D])|/2$. By Theorem \ref{thm:grotzcsh_thm}, every triangle-free planar graph is $3$-colorable so, in particular,
    \begin{align*}
        \oct(G - D) &\le |\bigcup_{C \in \mathcal{B}}C|/3 \\
        &\le 5|E(G[D])|/6 \textrm{ by the previous observation}\\
        &\le 5|D|/3 \textrm{ since $G$ is triangle-free}
    \end{align*}
    as desired.
\end{proof}

Let $f_g(\Delta)$ denote the smallest integer $k$ such that every planar graph of girth at least $g$ and maximum degree $\Delta$ is $2$-distance $k$-colorable. As proven in \cite{amini2013unified} and in~\cite{havet2017listcolouringsquaresplanar}, $f_3(\Delta) \le 3(1 + o(1))\Delta/2$ and as proven in \cite{BONAMY2019218}, $f_5(\Delta) \le \Delta + 2$.

\begin{lemma}\label{lemma:connected_dominant_kind_of}
    Let $G$ be a graph and let $S \in V(G)$. There exists $S' \in V(G)$ and $S_1, \dots, S_k$ for some $k \le |S|$, a partition of $S'$ such that
    \begin{itemize}
        \item $S \subseteq S'$, and
        \item for every $i$, $S_i$ is the vertex set of a connected component of $G[S']$, and
        \item for every $i$, $j$, $\dist_G(S_i, S_j) \ge 4$, and
        \item $|S'| \le 3|S|$.
    \end{itemize} 
\end{lemma}

\begin{proof}
    Let $I = \{1, \dots k\}$ and let $\{S_i, i\in I\}$ be a partition of $S$ such that for every $i \in I$, $|S_i| = 1$. We will repeat the following operation until a fixed point is reached : if there exists $i$, $j$ such that $\dist_G(S_i, S_j) \le 3$, then let $P$ be a path in $G$ of length at most $3$ with one endpoint in $S_i$ and the other one in $S_j$. Then, we redefine $S_i$ to be $S_i \cup S_j \cup P$, we delete $S_j$ and we remove $j$ from $I$. At each step of the procedure, the number of set decreases by one and the size of the union of the remaining sets increases by at most $2$, therefore the procedure terminates and at the end of the procedure, $\sum_{i \in I} |S_i| \le 3|S|$ and we just have to set $S'$ to be $\bigcup_{i \in I}S_i$
\end{proof}

\begin{theorem}\label{thm:2_k-colo_planar_girth_4}
    Let $\mathcal{P}_4$ be the class of planar graphs of girth $4$. Then $\ex{2}{\mathcal{P}_4}{(n)} \le 8\sqrt{3}(1 + o(1))\sqrt{n}$.
\end{theorem}

\begin{proof}
    
    Let $S$ be the set of vertices of degree at least $N$ in $G$, $N$ to be determined later. Since $G$ is planar and is of girth at least $4$, $|S| \le 4n/N$. We start by constructing the sets $S'$ and $S_1, \dots, S_k$ from $S$ as in Lemma \ref{lemma:connected_dominant_kind_of}. By Lemma \ref{lemma:oct_bound}, for every $i \in \{1, \dots, k\}$, $\oct(G[N_G(S_i)]) \le 5|S_i|/3$ so $G[N_G[S_i]]$ is $\colo{2}{8|S_i|/3}$-colorable by coloring each vertex in $S_i$ with a unique distance-2 color, each vertex in a minimum odd cycle transversal of $N_G(S_i)$ with a unique distance-2 color as well, and $2$-coloring the remaining vertices. Moreover, for every $i$, $j \in \{1,\dots,k\}$, $\dist_G(N_G[S_i], N_G[S_j]) \ge 2$ so we can $\colo{2}{8|S_i|/3}$-color each $N_G[S_i]$ to obtain a valid $\colo{2}{8|S|}$-coloring of $G[N_G[S]]$ since $\sum_{i \in \{1, \dots, k\}} |S_i| \le 3|S|$. Therefore, by Remark \ref{rem:colo_strat}, and since $G - N_G[S]$ is of maximum degree at most $N - 1$, $G$ is $\colo{2}{f_4(N - 1) + 8|S|}$-colorable. We can upper-bound $f_4(N - 1) + 8|S|$ by $3(1 + o(1))N/2 + 32n/N$, and by taking $N = 8/\sqrt{3(1 + o(1))}\sqrt{n}$, we use $8\sqrt{3}(1 + o(1))\sqrt{n}$ distance-2 colors as desired.
\end{proof}

\begin{figure}
    \centering
    \includegraphics[width=0.8\linewidth]{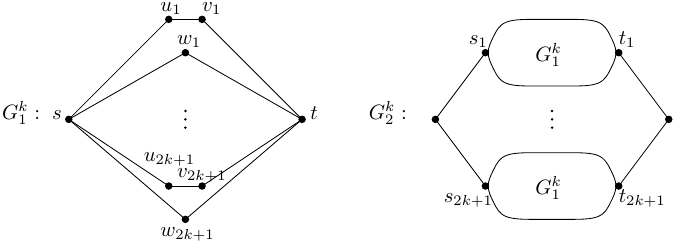}
    \caption{A planar graph of girth $4$ of order $12k^2 + 16k + 7$ not $\colo{2}{k}$-colorable.}
    \label{fig:planar_girth_4_sqrt_bound}
\end{figure}

\begin{figure}
    \centering
    \includegraphics[width=0.95\linewidth]{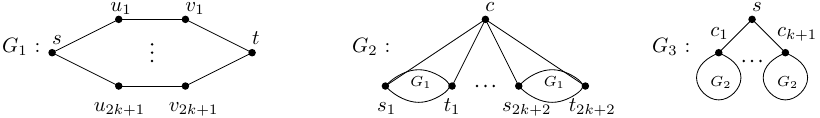}
    \caption{A planar graph of girth $5$ and order $8k^3 + 24k^2 + 25k + 10$ not $\colo{2}{k}$-colorable.}
    \label{fig:planar_girth_5_thirdsqrt_bound}
\end{figure}

The bound provided by Theorem \ref{thm:2_k-colo_planar_girth_4} is tight up to a constant factor since the graph $G_2^k$ from Figure \ref{fig:planar_girth_4_sqrt_bound} is of order $12k^2 + 16k + 7$ and is not $\colo{2}{k}$-colorable. Indeed, assume otherwise, and consider the graph $G_1^k$ from Figure \ref{fig:planar_girth_4_sqrt_bound}. If $s$ and $t$ are colored with a distance-1 color, then either they are colored with the same distance-1 color, but then since they both are of degree $4k + 2$, there must exist $i$ such that $u_i$ and $v_i$ are colored with a distance-1 color as well, in which case we must have a monochromatic edge. Otherwise, $s$ and $t$ are not colored with the same distance-1 color, but the same argument provides the existence of $i$ such that $w_i$ is colored with a distance-1 color, and again, we must have a monochromatic edge. So in any $\colo{2}{k}$-coloring of $G_1^k$, $s$ and $t$ cannot be both colored with a distance-1 color. However, in $G_2^k$, there must exist $i$ such that $s_i$ and $t_i$ are colored with a distance-1 color, which is a contradiction. This provides the lower-bound $\ex{2}{\mathcal{P}_4}{(n)} \ge \sqrt{3n-5}/6 - 1$. 

For the cases of planar graphs of larger girth, the same argument provides the bounds $\ex{2}{\mathcal{P}_g}{(n)} = O(\sqrt{n})$. Notice that Figure \ref{fig:planar_girth_5_thirdsqrt_bound} provides the lower-bound $\ex{2}{\mathcal{P}_5}{(n)} = \Omega(n^{1/3})$, the construction of Figure $2$ from \cite{choi2025propersquarecoloringssparse} provides the lower-bound $\ex{2}{\mathcal{P}_6}{(n)} = \Omega(n^{1/4})$, Theorem \ref{thm:choi_and_liu} states that $\ex{2}{\mathcal{P}_7}{(n)} \le 12$ and the construction of Figure $3$ from \cite{choi2025propersquarecoloringssparse} provides the lower-bound $\ex{2}{\mathcal{P}_7}{(n)} \ge 2$.

\subsubsection{Bounds on $\ex{3}{\mathcal{G}}{}$}

In this section, we will provide an upper-bound on the required number of distance-2 colors to color every planar graph of order $n$ when $3$ distance-1 colors are available.

\begin{remark}
    Let $G$ be a planar graph and let $v \in V$, $G[N(v)]$ is an outerplanar graph.
\end{remark}

\begin{theorem}\label{thm:bound_sqrt_planar}
    Let $\mathcal{P}$ be the class of planar graphs. Then $\ex{3}{\mathcal{P}}{(n)} \le 6\sqrt{3}(1 + o(1))\sqrt{n}$.
\end{theorem}

\begin{proof}
    Let $S$ be the set of vertices of $G$ of degree at least $N$, $N$ to be determined later. We start by constructing the sets $S'$ and $S_1, \dots, S_k$ from $S$ as in Lemma \ref{lemma:connected_dominant_kind_of}. Since for every $i$, $j$, $\dist_G(S_i, S_j) \ge 4$, then for every $i$, $j$, $\dist_G(N_G[S_i], N_G[S_j]) \ge 2$ so $G[N_G(S')]$ is an outerplanar graph and therefore is $3$-colorable. Since we can $2$-distance $f_3(N - 1)$-color $G - S$ and $\colo{3}{|S'|}$-color $G[N_G[S']]$ such that the vertices in $N_G(S')$ are only colored with distance-1 colors, the by Remark \ref{rem:colo_strat}, $G$ is $\colo{3}{f_3(N - 1) + |S'|}$-colorable. By Lemma \ref{lemma:connected_dominant_kind_of}, $|S'| \le 3|S|$ and since $G$ is planar, $|S| \le 6n/N$. Therefore, we can upper-bound $f_3(N - 1) + |S'|$ by $3(1 + o(1))N/2 + 18n/N$, and by taking $N = 6\sqrt{n}/\sqrt{3(1 + o(1))}$, we manage to use at most $6\sqrt{3}(1 + o(1))\sqrt{n}$ distance-2 colors as desired.
\end{proof}

The bound on $\ex{3}{\mathcal{P}}{(n)} \leq 6\sqrt{3}(1 + o(1))\sqrt{n}$ provided by Theorem~\ref{thm:bound_sqrt_planar} is tight up to a constant factor since the graph from Figure~\ref{fig:k-degenerate graph not 2k-colorable} with $k = 3$ provides the lower bound $\ex{3}{\mathcal{P}}{(n)}\geq (\sqrt{12n - 11} - 7)/6$.

\section{Further remarks and open problems}

\paragraph{Gap statements.} In Theorem \ref{thm:npc_gap_3_1}, we proved that distinguishing the $\colo{3}{1}$-colorable graphs and the $3$-colorable graphs is NP-Complete, even when restricted to planar graphs. We can prove a similar statement with an arbitrarily large gap, but for general graphs.

\begin{proposition}\label{prop:gap_c_0_c_k}
    For every integer $c \geq 3$, $k \ge 1$,  and given a graph $G$ that is either $c$-colorable or not $\colo{c}{k}$-colorable, it is NP-Complete to decide whether $G$ is $c$-colorable.
\end{proposition}

\begin{proof}
    Let $G'$ be the graph constructed from $G$ such that for every vertex $v \in V(G)$, we add $k + 1$ copies $v_1, \dots, v_{k + 1}$ of $v$ to $G'$ and for every edge $uv \in E(G)$  we add the edge $u_iv_j$ to $G'$ for every $i$, $j \in \{1, \dots, k + 1\}$. That is, we replace every edge of $G$ by a copy of $K_{k + 1, k + 1}$.
    \begin{itemize}
        \item if $G$ is $c$-colorable, then so is $G'$ since, for every proper $c$-coloring of $G$, we can color each $v_i$ with the color of $v$ and this is a proper $c$-coloring of $G'$.
        \item if $G'$ is $\colo{c}{k}$-colorable, then let $uv \in E(G)$. There must exist $i$, $j \in \{1, k + 1\}$ such that $u_i$ and $v_j$ are colored with distance-1 colors. Since $u_iv_j  \in E(G')$, those distance-1 colors are distinct. We can therefore recolor every copy of $u$ with the color of $u_i$ and every copy of $v$ with the color of $v_j$. By repeating this operation, we can transform the $\colo{c}{k}$-coloring of $G'$ into a proper $c$-coloring of $G'$ where each copies of $v$ shares the same color, and we can easily retrieve a proper $c$-coloring of $G$ from this proper $c$-coloring of $G'$. 
    \end{itemize}
    So $G$ is $c$-colorable if and only if $G'$ is $\colo{c}{k}$-colorable if and only if $G'$ is $c$-colorable. The result follows from the NP-Completeness of {\sc$c$-coloring} when $c \geq 3$. 
\end{proof}

\begin{question}
    Does a similar result restricted to planar graphs exist ?
\end{question}

\paragraph{Extremal values.}

As observed after Theorem \ref{thm:2_k-colo_planar_girth_4}, Figure \ref{fig:planar_girth_5_thirdsqrt_bound} provides the lower-bound $\ex{2}{\mathcal{P}_5}{(n)} = \Omega(n^{1/3})$. Moreover, the construction of Figure $2$ from \cite{choi2025propersquarecoloringssparse} provides the lower-bound $\ex{2}{\mathcal{P}_6}{(n)} = \Omega(n^{1/4})$.

\begin{question}
    Are these bounds tight ?
\end{question}

In most of our bounds, we prove that a graph $G$ of order $n$ is $\colo{\chi(G) - 1}{O(\sqrt{n})}$-colorable. This is false in general since, for instance, the graph from Figure \ref{fig:bound_bipartitr_planar_linear_1_k} is $2$-colorable but not $\colo{1}{o(n)}$-colorable. Another counter-example is a graph $G$ that is a complete $k$-partite graph with parts of size $n$. $G$ is $k$-colorable but is not $\colo{k - 1}{n}$-colorable. 

\begin{question}
    When is a graph $G$ of order $n$ $\colo{\chi(G) - 1}{O(\sqrt{n})}$-colorable ?
\end{question}


\section*{Acknowledgment}
I am sincerely grateful to Pascal Ochem for encouraging me to work on this topic and for his guidance and helpful suggestions throughout this work.

\bibliography{bibi}

@article{choi2025propersquarecoloringssparse,
  title={Between proper and square colorings of sparse graphs},
  author={Ilkyoo Choi and Xujun Liu},
  journal={\arxiv{2509.03080}},
  year={2025}
}

@article{DJPSY,
title = {The complexity of multiterminal cuts},
journal = {SIAM J. Comput.},
volume = {23},
number = {4},
pages = {864--894},
year = {1994},
issn = {},
doi = {},
url = {},
author = {Dahlhaus, Elias and Johnson, David S. and Papadimitriou, Christos H. and Seymour, Paul D. and Yannakakis, Mihalis}
}

@article{GAREY1976237,
title = {Some simplified NP-complete graph problems},
journal = {Theoretical Computer Science},
volume = {1},
number = {3},
pages = {237-267},
year = {1976},
issn = {0304-3975},
doi = {https://doi.org/10.1016/0304-3975(76)90059-1},
url = {https://www.sciencedirect.com/science/article/pii/0304397576900591},
author = {M.R. Garey and D.S. Johnson and L. Stockmeyer}
}

@ARTICLE{montassier2013near,
  title     = "Near-colorings: Non-colorable graphs and {NP-completeness}",
  author    = "Montassier, M and Ochem, P",
  journal   = "Electron. J. Comb.",
  publisher = "The Electronic Journal of Combinatorics",
  volume    =  22,
  number    =  1,
  month     =  mar,
  year      =  2015
}

@article{FEDER2021103210,
title = {Distance-two colourings of Barnette graphs},
journal = {European Journal of Combinatorics},
volume = {91},
pages = {103210},
year = {2021},
note = {Colorings and structural graph theory in context (a tribute to Xuding Zhu)},
issn = {0195-6698},
doi = {https://doi.org/10.1016/j.ejc.2020.103210},
url = {https://www.sciencedirect.com/science/article/pii/S0195669820301311},
author = {Tomás Feder and Pavol Hell and Carlos Subi}
}

@article{appel1976every,
  title={Every planar map is four colorable},
  author={Appel, Kenneth and Haken, Wolfgang},
  year={1976}
}

@article{appel1977every,
  title={Every planar map is four colorable. Part II: Reducibility},
  author={Appel, Kenneth and Haken, Wolfgang and Koch, John},
  journal={Illinois Journal of Mathematics},
  volume={21},
  number={3},
  pages={491--567},
  year={1977},
  publisher={Duke University Press}
}

@Inbook{Karp1972,
author="Karp, Richard M.",
editor="Miller, Raymond E.
and Thatcher, James W.
and Bohlinger, Jean D.",
title="Reducibility among Combinatorial Problems",
bookTitle="Complexity of Computer Computations: Proceedings of a symposium on the Complexity of Computer Computations, held March 20--22, 1972, at the IBM Thomas J. Watson Research Center, Yorktown Heights, New York, and sponsored by the Office of Naval Research, Mathematics Program, IBM World Trade Corporation, and the IBM Research Mathematical Sciences Department",
year="1972",
publisher="Springer US",
address="Boston, MA",
pages="85--103",
isbn="978-1-4684-2001-2",
doi="10.1007/978-1-4684-2001-2_9",
url="https://doi.org/10.1007/978-1-4684-2001-2_9"
}

@inproceedings{sharp2007distance,
  title={Distance coloring},
  author={Sharp, Alexa},
  booktitle={European Symposium on Algorithms},
  pages={510--521},
  year={2007},
  organization={Springer}
}

@article{wegner1977graphs,
  title={Graphs with given diameter and a coloring problem},
  author={Wegner, Gerd},
  year={1977}
}

@article{amini2013unified,
  title={A unified approach to distance-two colouring of graphs on surfaces},
  author={Amini, Omid and Esperet, Louis and Van Den Heuvel, Jan},
  journal={Combinatorica},
  volume={33},
  number={3},
  pages={253--296},
  year={2013},
  publisher={Springer}
}

@article{havet2017listcolouringsquaresplanar,
title = {List Colouring Squares of Planar Graphs},
journal = {Electronic Notes in Discrete Mathematics},
volume = {29},
pages = {515-519},
year = {2007},
note = {European Conference on Combinatorics, Graph Theory and Applications},
issn = {1571-0653},
doi = {https://doi.org/10.1016/j.endm.2007.07.079},
url = {https://www.sciencedirect.com/science/article/pii/S157106530700162X},
author = {Frédéric Havet and Jan {van den Heuvel} and Colin McDiarmid and Bruce Reed}
}

@article{bonamy20142,
  title={2-Distance Coloring of Sparse Graphs},
  author={Bonamy, Marthe and L{\'e}v{\^e}que, Benjamin and Pinlou, Alexandre},
  journal={Journal of Graph Theory},
  volume={77},
  number={3},
  pages={190--218},
  year={2014},
  publisher={Wiley Online Library}
}

@article{la20252,
  title={2-distance 4-coloring of planar subcubic graphs with girth at least 21},
  author={La, Hoang and Montassier, Mickael},
  journal={Discrete Mathematics \& Theoretical Computer Science},
  volume={26},
  number={Graph Theory},
  year={2025},
  publisher={Episciences. org}
}

@article{Wang_and_Lih,
author = {Wang, Wei-Fan and Lih, Ko-Wei},
title = {Labeling Planar Graphs with Conditions on Girth and Distance Two},
journal = {SIAM Journal on Discrete Mathematics},
volume = {17},
number = {2},
pages = {264-275},
year = {2003},
doi = {10.1137/S0895480101390448},
URL = {https://doi.org/10.1137/S0895480101390448}
}

@article{Borodin2004,
author = {Borodin, O.V. and Glebov, Aleksey and Ivanova, A.O. and Neustroeva, T.K. and Tashkinov, V.A.},
year = {2004},
month = {01},
pages = {},
title = {Sufficient conditions for planar graphs to be 2-distance {$(\Delta + 1)$}-colorable},
volume = {1},
journal = {Sibirskie Èlektronnye Matematicheskie Izvestiya [electronic only]}
}

@article{Borodin2004_2,
author = {Borodin, O.V. and Ivanova, A.O. and Neustroeva, T.K.},
journal = {Sibirskie Ehlektronnye Matematicheskie Izvestiya [electronic only]},
language = {eng},
pages = {76-90},
publisher = {Institut Matematiki Im. S.L. Soboleva, SO RAN},
title = {2-distance coloring of sparse planar graphs.},
url = {http://eudml.org/doc/51860},
volume = {1},
year = {2004},
}

@article{DVORAK2008838,
title = {Coloring squares of planar graphs with girth six},
journal = {European Journal of Combinatorics},
volume = {29},
number = {4},
pages = {838-849},
year = {2008},
note = {Homomorphisms: Structure and Highlights},
issn = {0195-6698},
doi = {https://doi.org/10.1016/j.ejc.2007.11.005},
url = {https://www.sciencedirect.com/science/article/pii/S0195669807002028},
author = {Zdeněk Dvořák and Daniel Král’ and Pavel Nejedlý and Riste Škrekovski},
}

@article{DVORAK20092634,
title = {Distance constrained labelings of planar graphs with no short cycles},
journal = {Discrete Applied Mathematics},
volume = {157},
number = {12},
pages = {2634-2645},
year = {2009},
note = {Second Workshop on Graph Classes, Optimization, and Width Parameters},
issn = {0166-218X},
doi = {https://doi.org/10.1016/j.dam.2008.08.013},
url = {https://www.sciencedirect.com/science/article/pii/S0166218X08003466},
author = {Zdeněk Dvořák and Daniel Král’ and Pavel Nejedlý and Riste Škrekovski},
}

@article{BONAMY2019218,
title = {Planar graphs of girth at least five are square {$(\Delta+2)$}-choosable},
journal = {Journal of Combinatorial Theory, Series B},
volume = {134},
pages = {218-238},
year = {2019},
issn = {0095-8956},
doi = {https://doi.org/10.1016/j.jctb.2018.06.005},
url = {https://www.sciencedirect.com/science/article/pii/S0095895618300522},
author = {Marthe Bonamy and Daniel W. Cranston and Luke Postle}
}

@article{ROBERTSON19972,
title = {The Four-Colour Theorem},
journal = {Journal of Combinatorial Theory, Series B},
volume = {70},
number = {1},
pages = {2-44},
year = {1997},
issn = {0095-8956},
doi = {https://doi.org/10.1006/jctb.1997.1750},
url = {https://www.sciencedirect.com/science/article/pii/S0095895697917500},
author = {Neil Robertson and Daniel Sanders and Paul Seymour and Robin Thomas},
}

@article{inoue20255,
  title={5-Coloring Planar Graphs with a Color Class of Order at Most $|V|/6$},
  author={Inoue, Yuta and Kawarabayashi, Ken-ichi and Miyashita, Atsuyuki},
  journal={\arxiv{2510.15407}},
  year={2025}
}

@ARTICLE{cranston2016planargraphsindependenceratio,
  title     = "Planar Graphs have Independence Ratio at least 3/13",
  author    = "Cranston, Daniel W and Rabern, Landon",
  journal   = "Electron. J. Comb.",
  publisher = "The Electronic Journal of Combinatorics",
  volume    =  23,
  number    =  3,
  month     =  sep,
  year      =  2016
}

@article{grunbaum1963grotzsch,
  title={Gr{\"o}tzsch's theorem on $3 $-colorings.},
  author={Gr{\"u}nbaum, Branko},
  journal={Michigan Mathematical Journal},
  volume={10},
  number={3},
  pages={303--310},
  year={1963},
  publisher={University of Michigan, Department of Mathematics}
}

@article{thomassen2003short,
  title={A short list color proof of Gr{\"o}tzsch's theorem},
  author={Thomassen, Carsten},
  journal={Journal of Combinatorial Theory, Series B},
  volume={88},
  number={1},
  pages={189--192},
  year={2003},
  publisher={Elsevier}
}

@article{kramer1969probleme,
  title={Un probleme de coloration des sommets d’un graphe},
  author={Kramer, Florica and Kramer, Horst},
  journal={CR Acad. Sci. Paris A},
  volume={268},
  number={7},
  pages={46--48},
  year={1969}
}

@article{krumke2001models,
  title={Models and approximation algorithms for channel assignment in radio networks},
  author={Krumke, Sven O and Marathe, Madhav V and Ravi, SS},
  journal={Wireless networks},
  volume={7},
  number={6},
  pages={575--584},
  year={2001},
  publisher={Springer}
}

@article{cranston2022coloring,
  title     = "Coloring, list coloring, and painting squares of graphs (and
               other related problems)",
  author    = "Cranston, Daniel",
  journal   = "Electron. J. Comb.",
  publisher = "The Electronic Journal of Combinatorics",
  volume    =  1000,
  number    = "DS25",
  month     =  apr,
  year      =  2023
}

@article{pinlou2006oriented,
  title={Oriented vertex and arc colorings of outerplanar graphs},
  author={Pinlou, Alexandre and Sopena, {\'E}ric},
  journal={Information Processing Letters},
  volume={100},
  number={3},
  pages={97--104},
  year={2006},
  publisher={Elsevier}
}
\bibliographystyle{plain}

\end{document}